\documentclass[a4paper, 12pt]{amsart}
\usepackage[top=2cm, bottom=2cm, left = 2cm, right = 2cm]{geometry} 
\usepackage{amsthm,amsmath,amscd,amssymb,amsfonts,graphicx,mathrsfs}
\usepackage{pdfsync}
\usepackage[all]{xy}
\usepackage{tikz-cd}
\usepackage{color}
\usepackage{mathtools}
\usepackage{microtype}
\usepackage{textcomp}
\usepackage{booktabs}
\usepackage{tabularx}
\usepackage[english]{babel} 
\usepackage[utf8]{inputenc}
\usepackage{bbm}
\usepackage{graphicx}
\usepackage{adjustbox}
\usepackage{longtable}
\usepackage[table]{xcolor}
\usepackage{enumitem}
\title{Cylinders in weighted Fano varieties}

\author{Adrien Dubouloz}
\address{Laboratoire de Math\'ematique et Application, UMR 7348 CNRS, Universit\'e de Poitiers, 86000 Poitiers, France \\
Universit\'e de Bourgogne Europe, CNRS, IMB UMR 5584, 21000 Dijon, France}
\email{adrien.dubouloz@math.cnrs.fr}

\author{In-Kyun Kim}
\address{June E Huh Center for Mathematical Challenges, Korea Institute for Advanced Study, 85, Hoegiro Dongdaemun-gu, Seoul 02455, Republic of Korea }
\email{soulcraw@kias.re.kr} 

\author{Takashi Kishimoto}
\address{Department of Mathematics, Faculty of Science, Saitama University, Saitama 338-8570, Japan}
\email{kisimoto.takasi@gmail.com}

\author{Joonyeong Won}
\address{Department of Mathematics, Ewha Womans University, 52, Ewhayeodae-gil, Seodaemun-gu, Seoul, 03760, Republic of Korea} 
\email{leonwon@ewha.ac.kr}

\newtheorem{theorem}{Theorem}[section]
\newtheorem{definition}[theorem]{Definition}
\newtheorem{corollary}[theorem]{Corollary}
\newtheorem{proposition}[theorem]{Proposition}
\newtheorem{lemma}[theorem]{Lemma}
\newtheorem{question}[theorem]{Question}
\newtheorem{conjecture}[theorem]{Conjecture}
\theoremstyle{definition}
\newtheorem{example}[theorem]{Example}

\theoremstyle{remark}
\newtheorem{remark}[theorem]{Remark}
\newtheorem{notation}[theorem]{Notation}

\newcommand{\Sing}{\operatorname{Sing}}

\newcommand{\mult}{\operatorname{mult}}

\newcommand{\PP}{{\mathbb P}}

\newcommand{\msp}{\mathsf p}

\begin{document}

\begin{abstract}
Cylinders in Fano varieties receives a lot of attentions recently from the viewpoints of birational geometry and unipotent geometry. In this article, we provide a survey of several known and new results concerning the anti-canonically polar cylindricity of quasi-smooth, well-formed
weighted Fano complete intersections in weighted projective spaces. 
\end{abstract}

\maketitle
\vspace{-1em}
\section*{Introduction}
A {\it cylinder} in a normal projective variety $X$ defined over a field $\Bbbk$ of characteristic zero is a nonempty Zariski open subset of $X$ isomorphic to $Z\times_{\Bbbk} \mathbb{A}^1_{\Bbbk}$ for some affine $\Bbbk$-variety $Z$. A variety containing a cylinder is called cylindrical. 
Cylindrical normal projective varieties receive a lot of attentions for the last decade in particular due to the connection between cylinders in them and unipotent group actions on their generalized affine cones (cf. \cite{KPZ1}, \cite{KPZ3}). Every smooth rational projective surface is cylindrical, but cylindricity  is not a birational invariant, even for surfaces with mild singularities. For instance 
a del Pezzo surface $S$ of degree $1$ with canonical singularities $2D_4$, $2A_3 + 2A_1$ or $4A_2$ is not cylindrical (cf. \cite{CPW,Saw}), but its minimal resolution is cylindrical since it is rational. 

A smooth cylindrical projective variety $X$ being in particular birationally $\mathbb{P}^1$-ruled, it follows from \cite{BCHM} that it has birational model $Y$ with mild singularities which is either a Fano variety or the total space of a Mori fiber space $f:Y\to S$ over a positive dimensional base $S$. 
The cylindricity of $X$ does not necessarily imply that of $Y$, but the converse holds by \cite{DK}. It is thus natural and important to understand under which condition the total space of a Mori fiber space $f: Y \to S$ is cylindrical. By \cite{DK4,DK5}, a sufficient condition is that the fiber $Y_\eta$ of $f$ over the generic point $\eta$ of $S$, which is Fano variety of Picard rank $1$ defined over the function field $\Bbbk (S)$ of $S$, is cylindrical over $\Bbbk (S)$. The base change of a cylindrical $\Bbbk$-variety $X$ to any field extension $\Bbbk'$ of $\Bbbk$ is again cylindrical, but the converse does not hold in general since for instance a smooth conic without $\Bbbk$-rational point is not cylindrical.  So even when working over algebraically closed  fields, the study of normal cylindrical projective varieties ultimately connects to the problem characterizing the cylindricity of mildly singular Fano varieties defined over non-necessarily algebraically closed fields and of arbitrary Picard ranks. 
The main purpose of this article is to provide a survey of several techniques to study the cylindricity of Fano varieties illustrated by many examples, with a particular focus on the class of quasi-smooth, well-formed weighted Fano complete intersections in weighted projective spaces.  Section 1 presents on the one hand basic examples of cylindrical projective varieties and on the other hand a selection of nowadays classical tools that can be used to show that certain classes of mildly singular Fano
varieties do not contain any cylinder. Section 2 is devoted to a recollection on weighted projective spaces and their quasi-smooth, well formed subvarieties. In section 3, we review the current of knowledge concerning the cylindricity of del Pezzo and Fano threefold hypersurfaces in weighted projective spaces. In the last section, we present several results concerning the cylindricity of Fano complete intersections of higher codimensions in weighted projectives spaces. 

\vspace{2mm}  
\noindent {\bf Acknowledgements}: The authors thank Saitama University for the excellent working condition offered during the preparation of this article. 
The third author was partially supported by JSPS KAKENHI Grant Number 23K03047. The second author was supported by the National Research Foundation of Korea under Grant number RS-2025-00513064, while the fourth author was supported under Grant numbers RS-2025-00513064 and RS-2026-25506097.


\section{Cylindrical projective varieties: basic examples and obstructions}

\begin{definition}\label{def:apc}
 Let $X$ be a projective variety over a field $\Bbbk$ and let $H$ be an ample Weil $\mathbb{Q}$-divisor on $X$. An {\it $H$-polar cylinder} in $X$ is a cylinder $U\cong Z\times_{\Bbbk} \mathbb{A}^1_{\Bbbk}$ in $X$ whose complement $X\setminus U$ is the support of an effective Weil $\mathbb{Q}$-divisor $D$ which is $\mathbb{Q}$-linearly equivalent to $H$. 

An {\it anti-canonically polar cylinder} on a Fano variety $X$ is a $(-K_X)$-polar cylinder, where $K_X$ denotes a canonical divisor of $X$.
\end{definition}
\subsection{Normal projective varieties with unipotent group actions are cylindrical} \label{section6}

\begin{proposition}\label{prop:unipotent}
A normal projective variety $X$ endowed with a non-trivial action of a unipotent group contains an $H$-polar cylinder with respect to every ample Weil $\mathbb{Q}$-divisor class $H$. 
\end{proposition}
\begin{proof}
The hypothesis implies in particular that $X$ admits a non-trivial action of the additive group $G=\mathbb{G}_{a,\Bbbk}$. Let $m>0$ be such that $mH$ is Cartier and very ample, let $\mathcal{L}=\mathcal{O}_X(mH)$ be the corresponding invertible sheaf, let $W=H^0(X,\mathcal{L})$ and let $\varphi:X \hookrightarrow \mathbb{P}(W)$ be the associated closed embedding. Since $X$ is normal, the invertible sheaf $\mathcal{L}$ admits a $G$-linearization. This follows for instance from \cite[Lemma 2.3]{KKLV89} together with the fact that since $X$ is normal, the pull-back homomorphism $\mathrm{pr}_X^*:\mathrm{Pic}(X)\to \mathrm{Pic}(G\times_{\Bbbk} X)$ is an isomorphism (cf. \cite{Mag76}). 

The choice of such a $G$-linearization provides a linear representation $\rho :G\to \mathrm{GL}(W)$ such that $\varphi$ is $G$-equivariant with respect to the induced $G$-action on $\mathbb{P}(W)$. By \cite[Proposition 3]{Sesh62}, the linear representation $\rho$ splits in the form $W=\bigoplus_{i=1}^{r} W_i$ where $W_i=\mathrm{Sym}^{n_i}V$ for a fixed $2$-dimensional $\Bbbk$-vector space $V$ viewed as an $\mathrm{SL}_{2,\Bbbk}$-module and the $G$-action on $W_i$ is induced by the action of  $\mathrm{SL}_{2,\Bbbk}$ on $W_i$ through the inclusion $G \hookrightarrow \mathrm{SL}_{2,\Bbbk}$ as the subgroup consisting of upper triangular unipotent matrices. 

Since the action of $G$ is non-trivial, there exists at least one $W_i$ of dimension $n_i+1\geq 2$, say $W_1$. The corresponding representation $G\to \mathrm{GL}(W_1)$ is then given by $G\ni t\mapsto \exp(tJ_1)$, where $J_1$ is the $(n_1+1) \times (n_1+1)$ Jordan matrix:  $$\left(\begin{array}{cccc}
0 & 1 & 0 & 0\\
\vdots & \ddots & \ddots & 0\\
\vdots & \ddots & \ddots & 1\\
0 & \cdots & \cdots & 0
\end{array}\right). $$
In particular, there exist homogeneous coordinates $[x_0:\ldots : x_n]$ on $\mathbb{P}(W)\cong \mathbb{P}^n_{\Bbbk}$ 
such that $x_0$ is $G$-invariant and $x_1$ is mapped to $x_1+tx_0$. The restriction of the $G$
-action to $\mathbb{P}(W)\setminus \{x_0=0\}\cong \mathbb{A}^n_{\Bbbk}$ is then a translation with the affine coordinate function $y=x_1/x_0$ as a slice. It follows in turn that the $G$-stable affine open subset $X_0=X\setminus \{x_0=0\}$ of $X$ is $G$-equivariantly isomorphic to $(X_0\cap\{y=0\})\times_{\Bbbk} \mathrm{Spec}(\Bbbk[y]) $ on which $G$ acts by translations on the second factor. In particular, $X_0$ is an $mH$-polar cylinder, whence an $H$-polar cylinder, in $X$. 
\end{proof}

\begin{corollary}    \label{prop:unipotent2}
A normal Fano variety endowed with a non-trivial action of a unipotent group  contains an anti-canonical polar cylinder. 
\end{corollary}

\begin{remark}\label{rem:unipotent2} Projective spaces and smooth quadric hypersurfaces admit effective action of unipotent groups, whence are (anti-canonically polar) cylindrical. But  many cylindrical Fano varieties do not admit non-trivial unipotent group actions. For instance every smooth del Pezzo surface is $\mathbb{A}^2$-cylindrical, but such a surface admits non-trivial action of a unipotent group if and only if it is of degree larger than $6$. See also Remark \ref{exa:non-unipotent-wps} for another $2$-dimensional example.

Another illustration in higher dimension is given by smooth prime Fano threefolds of degree $22$: all of these are  cylindrical \cite{KPZ1}, but only two members of the $6$-dimensional coarse moduli space of such threefolds admit a non-trivial action of the additive group ${\mathbb G}_a$: the Mukai-Umemura threefold $X_{22}^{\rm MU}$ with automorphism group  isomorphic to $\mathrm{PGL}_{2}$ and a threefold $X_{22}^a$ with automorphism group isomorphic to $\mathbb{G}_a\rtimes \mu_4$, see \cite{KPS}, \cite{DFK}. 
\end{remark}

\subsection{Obstructions to cylindricity}
In this section, we review existing techniques that can be used to detect the non-existence of cylinders in certain projective varieties. We first set up the notation for a construction which will be used several times in the rest of the article.

\subsubsection{Log-resolution of cylinders}\label{subsec:log-res-cyl}
Let $X$ be a normal projective variety and let $D=\sum_{i=1}^sd_i D_i$, where the $D_i$ are the irreducible components of $D$, be an effective Weil $\mathbb{Q}$-divisor such that the complement in $X$ of the support $\mathrm{Supp}(D)$ of $D$ contains a cylinder $Y\times_{\Bbbk} \mathbb{A}^1_{\Bbbk}$ for some normal affine variety $Y$. Let $Z$ be an irreducible Zariski dense open subset of the regular locus of $Y$ and let $\bar{Z}$ be a smooth projective completion of $Z$. The projection $\mathrm{pr}_Z:Z\times_{\Bbbk} \mathbb{A}^1_{\Bbbk}\to Z$ induces a rational map $\varphi:X\dashrightarrow \bar{Z}$. Let $\sigma:\tilde{X}\to X$ be a projective birational morphism restricting to an isomorphism over $Z\times_{\Bbbk} \mathbb{A}^1_{\Bbbk}$ and which is a simultaneous log-resolution of the pair $(X,D)$ and of the indeterminacy locus of $\varphi$. So, by definition, $\tilde{X}$ is a non-singular projective variety and the union of the support of proper transform $\tilde{D}=\sum_{i=1}^s d_i\tilde{D}_i$ of $D$ in $\tilde{X}$ with the exceptional locus $E=\bigcup_{i=1}^rE_i$ of $\sigma$, where the $E_i$ are the irreducible components of $E$, is an $SNC$ divisor on $\tilde{X}$.  
Since $\tilde{X}$ is regular, hence normal, the fiber $\tilde{X}_\eta$ of the induced projective morphism 
$\tilde{\varphi}=\varphi\circ \sigma:\tilde{X}\to \bar{Z}$
over the generic point $\eta$ of $\bar{Z}$ is a normal, whence regular, projective curve over the function field $K$ of $\bar{Z}$. On the other hand, since the restriction of $\tilde{\varphi}$ to $\sigma^{-1}(Z\times_{\Bbbk} \mathbb{A}^1_{\Bbbk})\cong Z\times_{\Bbbk} \mathbb{A}^1_{\Bbbk}$ is a trivial $\mathbb{A}^1$-bundle over its image, it follows that $\tilde{X}_\eta$ contains $\mathbb{A}^1_K$ as a Zariski open subset, and hence, that $\tilde{X}_\eta\cong \mathbb{P}^1_K$, where $K:=\Bbbk (\overline{Z})$ is the function field of $\overline{Z}$. It follows that $\tilde{\varphi}:\tilde{X}\to \bar{Z}$ is a $\mathbb{P}^1$-fibration, which further restricts to a trivial $\mathbb{P}^1$-bundle over a Zariski dense open subset $Z_0$ of $Z$. Moreover, the closure in $\tilde{X}$ of the $K$-rational point $\infty_\eta=\tilde{X}_\eta\setminus \mathbb{A}^1_K$ of $\tilde{X}_\eta$ is a prime Weil divisor $\tilde{H}_\infty$ on which $\tilde{\varphi}$ restricts to a dominant birational morphism $\alpha:\tilde{H}_\infty \to \bar Z$ which is an isomorphism over $Z_0$. 

\subsubsection{Pseudo-effectiveness of canonical divisors and cylindricity}\label{4-2-3}

\begin{proposition} \label{prop:pseff_cylinder} Let $X$ be a normal projective variety and let $D$ be an effective Weil $\mathbb{Q}$-divisor such that $X\setminus\mathrm{Supp}(D)$ contains a cylinder and $K_X+D$ is $\mathbb{Q}$-Cartier and pseudo-effective. Then the log pair $(X,D)$ is not log-canonical.
\end{proposition}
\begin{proof}
We can assume without loss of generality that $\Bbbk$ is algebraically closed. We let $U=Z\times_{\Bbbk} \mathbb{A}^1_{\Bbbk}$ be a cylinder with $Z$ non-singular contained in the complement of $\mathrm{Supp}(D)$ and we use the construction and notation of subsection \ref{subsec:log-res-cyl}. The birational section $\tilde{H}_\infty$ of $\tilde{\varphi}:\tilde{X}\to \bar{Z}$ 
is either the proper transform or an irreducible component of $D$, or the proper transform of an irreducible Weil divisor $H_\infty$ on $X$ not contained in the support of $D$, or is equal to one of the exceptional divisor $E_i$. 
The ramification formula for the $\mathbb{Q}$-Cartier divisor $K_X+D$ reads
$$ K_{\tilde{X}}+\sum_{i=1}^s d_i\tilde{D}_i=\sigma^*(K_X+D)+\sum_{j=1}^r a_jE_j.$$ For a general closed fiber $\ell\cong \mathbb{P}^1_{\Bbbk}$ of the $\mathbb{P}^1$-fibration $\tilde{\varphi}$, we have 
\begin{equation}\label{eq:ram_1}
\ell\cdot (K_{\tilde{X}}+\tilde{D})=\begin{cases}
-2+d_{i} & \textrm{if }\tilde{H}_{\infty}=\tilde{D}_{i}\textrm{ for some }i\in\{1,\ldots,s\}\\
-2 & \textrm{otherwise.}
\end{cases}
\end{equation}
On the other hand, we have 
\begin{equation} \label{eq:ram_2}
\ell\cdot (\sigma^*(K_X+D)+\sum_{i=1}^{r} a_iE_i)=(K_X+D)\cdot\sigma_*\ell+\begin{cases}
a_{j} & \textrm{if }\tilde{H}_{\infty}=E_{j}\textrm{ for some }j\in\{1,\ldots,r\}\\
0 & \textrm{otherwise.}
\end{cases}
\end{equation}
Since $(K_X+D)$ is pseudo-effective and the curves $\sigma_*\ell$ moves in an algebraic family, we have  $(K_X+D)\cdot\sigma_*\ell\geq 0$. This excludes in particular the possibility that $ \tilde{H}_\infty$ is the proper transform of a divisor on $X$ other than an irreducible component of $D$ since otherwise equations \eqref{eq:ram_1} and \eqref{eq:ram_2} would give the absurd inequality $-2\geq 0$.
So, either $\tilde{H}_\infty=\tilde{D}_i$ for some $i\in\{1,\ldots,s\}$ and then 
$d_i\geq 2$ which implies that the log pair $(X,D)$ is not log-canonical at every point of $D_i$ or $\tilde{H}_\infty=E_j$ for some $j\in\{1,\ldots,r\}$ and then  $a_j\leq -2$ 
and the log pair $(X,D)$ is not log-canonical at every point of the image $\sigma(E_i)$ of $E_i$ in $X$. 
\end{proof}    
\begin{corollary} \label{prop:cppz}\label{cor:Kpseff+lc_implies_nocylinder} A normal  projective variety with log-canonical singularities such that $K_X$ is 
pseudo-effective does not contain a cylinder.  
\end{corollary}
\begin{proof} This follows from Proposition \ref{prop:pseff_cylinder} applied to $D=\emptyset$. 
\end{proof}

\begin{remark} In Corollary \ref{cor:Kpseff+lc_implies_nocylinder}, the assumption on the singularities of $X$ cannot be weakened in general. For instance, let $Y\subset \mathbb{P}^n_{\Bbbk}$, $n\geq 2$ be any smooth hypersurface of degree $d\geq n+2$, let $X\subset \mathbb{P}^{n+1}_{\Bbbk}$ be the cone over $Y$ and let $\pi_{\{o\}}:X\dashrightarrow Y$ be the projection from the vertex $\{o\}$ of $X$. Then, by the adjunction formula, we have $$K_X=(K_{\mathbb{P}^{n+1}_{\Bbbk}}+X)|_X\sim (-n-2+d)H|_X,$$ where $H$ denotes a hyperplane of $\mathbb{P}^{n+1}_{\Bbbk}$, is Cartier and pseudo-effective (in fact, either trivial if $d=n+2$ or ample otherwise). The variety  $X$ is also cylindrical, more precisely, for every ample effective Weil $\mathbb{Q}$-divisor $D$ in $Y$, the complement in $X$ of the support of the cone $\hat{D}$ over $D$ is a cylinder, isomorphic to $(Y\setminus \mathrm{Supp}(D))\times_{\Bbbk} \mathbb{A}^1_{\Bbbk}$. 
But on the other hand, $X$ is not log-canonical at its unique singular point $\{o\}$.  Indeed, a log-resolution of $\{o\}$ is given by the blow-up $\sigma:\tilde{X}\to X$ with exceptional divisor $E\cong Y$ and in the ramification formula $K_{\tilde{X}}=\sigma^*K_X+aE$, taking intersection with a general fiber closed fiber $\ell\cong \mathbb{P}^1$ of the induced $\mathbb{P}^1$-bundle $\pi\circ \sigma:\tilde{X}\to Y$, yields $a=n-d\leq -2$.
\end{remark}

\begin{corollary} \label{cor:cyl_implies_notlc}Let $X$ be a normal projective variety and let $D$ be an effective $\mathbb{Q}$-Cartier Weil $\mathbb{Q}$-divisor $\mathbb{Q}$-linearly equivalent to $-K_X$ and such that $X\setminus\mathrm{Supp}(D)$ is a cylinder. Then the log pair $(X,D)$ is not log-canonical at a point of the support of $D$.  
\end{corollary}
\begin{proof}
Again, we can assume without loss of generality that the base field $\Bbbk$ is algebraically closed. Since $K_X+D\sim_{\mathbb{Q}} 0$, the argument in the proof of Proposition \ref{prop:pseff_cylinder} implies that either $\tilde{H}_\infty=\tilde{D}_i$ for some $i\in\{1,\ldots,s\}$ and then the log pair $(X,D)$ is not log-canonical at every point of $D_i$ or $\tilde{H}_\infty=E_j$ for some $j\in\{1,\ldots, r\}$ and then $\sigma(E_j)\subset X$ is a non log-canonical center of the log pair $(X,D)$. In the second case, we claim that $\sigma(E_j)\subset \mathrm{Supp}(D)$. Indeed,
since $X\setminus\mathrm{Supp}(D)\cong Y\times_{\Bbbk} \mathbb{A}^1_{\Bbbk}$ for some  variety $Y$, it follows from the definition of $\tilde{H}_\infty$ in subsection  \ref{subsec:log-res-cyl} that the composition $\sigma\circ\alpha^{-1}: \bar{Z}\dasharrow X$ induces a dominant rational map $\beta:\bar{Z}\dashrightarrow \sigma(E_j)=\sigma(\tilde{H_\infty})$ whose restriction to a Zariski dense open subset $Z$ of $Y$ associates to a closed point $z$ of $Z$ the unique intersection point of the closure in $X$ of $\mathrm{pr}_Y^{-1}(z)\cong \mathbb{A}^1_{\Bbbk}$ with $\sigma(E_j)$. But since by definition of a cylinder, $X\setminus\mathrm{Supp}(D)$ is an affine variety, which therefore does not contain any proper curve, it follows that the image of $\beta$ contains a Zariski dense open subset disjoint from $X\setminus\mathrm{Supp}(D)$, hence contained in $\mathrm{Supp}(D)$. Since $\sigma(E_i)$ is a Zariski closed subset of $X$, this implies in turn that $\sigma(E_i)\subset \mathrm{Supp}(D)$.   
\end{proof}

\subsubsection{Cylindricity and $\alpha$-invariant of Fano varieties}\label{4-2-1}
The log canonical threshold plays a central role in birational geometry. In particular, it gives rise to the notion of the $\alpha$-invariant, which has been proved to be a powerful tool in the study of K\"ahler–Einstein metrics on Fano varieties. Indeed, the $\alpha$-invariant often serves as a criterion for the existence of such metrics. Moreover, it has been proved to be a powerful tool in both affine and birational geometry. In particular, it can be used to detect the non-existence of anti-canonical polar cylinders found in Fano varieties. 
\begin{definition}\label{def:alpha}
Let $X$ be a normal projective variety such that $K_X$ is $\mathbb{Q}$-Cartier and let $H$ be a Weil $\mathbb{Q}$-Cartier $\mathbb{Q}$-divisor on $X$. The $\alpha$-invariant of the log pair $(X,H)$ is defined to be:
\begin{equation*}
    \alpha(X, H) = \sup \left\{\,\lambda \in \mathbb{Q} \,\middle|\, 
        \begin{array}{l}
            (X, \lambda D) \text{ is log canonical for all effective } \mathbb{Q}\text{-divisors } D \\
            \text{such that } D \sim_{\mathbb{Q}} H
        \end{array}
        \,\right\}.
    \end{equation*}
\end{definition}
In the case of $H=-K_X$, the invariant $\alpha(X)\coloneqq \alpha(X,-K_X)$ plays fundamental roles in various situations. A celebrated theorem of Tian \cite{Tian87} asserts that if $X$ is a smooth Fano variety of dimension $n$ satisfying: 
\begin{equation}\label{The inequality of alpha invariant}
    \alpha(X) >\frac{n}{n + 1}, 
\end{equation}
then $X$ admits a K\"ahler–Einstein metric. The $\alpha$-invariant of a Fano variety also allows to detect  the non-existence of anti-canonical polar cylinders in Fano varieties: roughly speaking, a "large" value of $\alpha$ prevents the existence of an anti-canonically polar cylinder. More precisely, we have the following: 
\begin{theorem}[{\cite[Theorem 1.26]{CPPZ}}]
\label{theorem:alpha cylinder}
    Let $X$ be a Fano variety with log-canonical singularities. If $\alpha(X)\geq 1$, then $X$ does not contain any anti-canonical polar cylinder.
\end{theorem}
\begin{proof} By Proposition \ref{prop:pseff_cylinder}, if, for an effective Weil $\mathbb{Q}$-divisor $D\sim_{\mathbb{Q}} -K_X$, the corresponding complement $X\setminus \mathrm{Supp}(D)$ is a cylinder, then the log pair $(X,D)$ is not log-canonical, which amounts to $\alpha(X)<1$. 
\end{proof}

\begin{example}
Let $X\subset \mathbb{P}^n$ be a smooth hypersurface of degree $n$. If $n \geq 6$ and $X$ is general, then $\alpha(X)\geq 1$ (see {\cite[Theorem 2]{Pu05}}). Therefore, by Theorem \ref{theorem:alpha cylinder}, $X$ does not contain any anti-canonical polar cylinder. For $n=3$, however, we have $\alpha(X)<1$. For $n=4$ (resp. $n=5$) the known lower bound is $\alpha(X)\geq \frac{7}{9}$ (resp. $\alpha(X)\geq \frac{5}{6}$) (see \cite{CPW14}). These bounds are not enough to determine whether $X$ is non-cylindrical by means of Theorem \ref{theorem:alpha cylinder}, so different arguments are required. 
More precisely, see \cite{CPW0} for $n=3$, meanwhile all smooth quartic threefolds and all smooth quintic fourfolds are known to be birationally superrigid (see \cite{IM}, \cite{Pu87}, \cite{TF13} and \cite{TF16}), which implies the non-cylindricity.     
\end{example}

\subsection{$K$-stability and anti-canonical polar cylindricity of Fano varieties}

The $\alpha$-invariant $\alpha(X)$ of a Fano variety $X$ has a strong connection with another invariant, called its $\delta$-invariant $\delta(X)$, which plays a central role in the study of $K$-stability of Fano varieties.  We only quote the following result, referring the reader to \cite{ACCF23} and \cite{BJ20} for comprehensive discussion on $K$-stability of Fano varieties and $\delta$-invariants.
\begin{theorem}[{\cite[Theorem B]{BJ20}}]\label{thm:BJ} For a Fano variety $X$ with Kawamata log terminal singularities, the following hold:
    \begin{itemize}
        \item $X$ is $K$-semistable if and only if $\delta(X)\geq 1$;
        \item $X$ is $K$-stable if and only if $\delta(X)>1$.
    \end{itemize}
\end{theorem}
By \cite[Theorem A]{BJ20}, one has the following inequalities between the $\alpha$-invariant and the $\delta$-invariant of a Fano variety $X$ with Kawamata log terminal singularities:
\begin{equation}\label{eq:alpha_delta}
    \left(\frac{\dim X + 1}{\dim X}\right)\alpha(X) \leq \delta(X) \leq (\dim X + 1)\alpha(X).
\end{equation}

In view of these inequalities, a "large enough" value of $\alpha$ is simultaneously an obstruction for the existence of an anti-canonical polar cylinder in a Fano variety $X$ by Theorem \ref{theorem:alpha cylinder} and a sufficient condition for the $K$-semistability of $X$ by Theorem \ref{thm:BJ}. For instance, one has the following consequence of Theorem \ref{theorem:alpha cylinder} and inequality \eqref{eq:alpha_delta}:

\begin{corollary}
    Let $X$ be a Fano variety with Kawamata log terminal singularities. If $\alpha(X)\geq 1$ then $X$ is $K$-stable and does not contain any anti-canonical polar cylinder.
\end{corollary}

On the other hand, it is a natural expectation that an anti-canonical polar cylinder in a Fano variety $X$ could yield a destabilizing test configuration for $X$, thereby obstructing its $K$-(poly)stablity.  This motivated the following conjectural connection:
\begin{conjecture}[{\cite[Conjecture 1.32]{CPPZ}}]\label{Conj:CPPZ}
Let $X$ be a Fano variety that has at most Kawamata log terminal singularities.
If $X$ does not contain $(-K_X)$-polar cylinders, then $X$ is $K$-polystable.
\end{conjecture}
It is known nowadays this conjecture fails \cite{KKW24}, see subsection \ref{subsec:NonCyl-dP}  below for examples of del Pezzo surfaces with quotient singularities which do not contain anti-canonical polar cylinders but are nevertheless $K$-unstable.

\section{Weighted complete intersections in weighted projective spaces}\label{section2} 
\subsection{Weighted projective spaces: definition and basic properties}\label{2-1}
We begin by recalling the basic definitions of several notions related to weighted projective spaces. For further details, the reader is referred to \cite{Dol81} and \cite{IF00}. Let $\Bbbk$ be a field of characteristic zero, which is not necessarily algebraically closed throughout this section. 

In what follows, we consider the polynomial ring $\Bbbk [x_0, \cdots , x_n]$ in $(n+1)$-variables over $\Bbbk$ as endowed with a structure of $\mathbb{N}$-graded $\Bbbk$-algebra determined by assigning the weight $a_i >0$ to each variable $x_i$ for $0 \leqq i \leqq n$. In order to distinguish it from the usual (non-graded) polynomial ring, we denote by
\[
\Bbbk (a_0, \ldots , a_n)=\bigoplus_{m\geq 0}\Bbbk(a_0,\ldots,a_n)_m 
\]
the $\Bbbk$-algebra $\Bbbk [x_0, \cdots , x_n]$ considered as a graded $\Bbbk$-algebra for the above $\mathbb{N}$-grading. The weighted projective space associated to this choice of grading is then defined as the $\Bbbk$-scheme: 
\[
\PP_{\Bbbk} (a_0, \ldots , a_n) \coloneqq {\rm Proj}_{\Bbbk} \Bbbk (a_0, \ldots , a_n). 
\]

Under the well known one-to-one correspondence between $\mathbb{Z}$-gradings on a $\Bbbk$-algebra $A$ of finite type and algebraic actions of the multicative $\Bbbk$-groups scheme $\mathbb{G}_{m,\Bbbk}={\rm Spec}(\Bbbk[t^{\pm 1}])$ on its spectrum, the above grading corresponds to a linear $\mathbb{G}_{m,\Bbbk}$-action $\sigma$ on $\mathbb{A}^{n+1}_{\Bbbk}$ given by the co-morphism \[\sigma^*: \Bbbk[x_0,\ldots, x_n]\to \Bbbk[x_0,\ldots, x_n]\otimes_{\Bbbk} \Bbbk[t^{\pm 1}]\cong \Bbbk[x_0,\ldots, x_n][t^{\pm 1}], \; x_i\mapsto x_it^{a_i}, \quad i=0,\ldots, n \]
and the weighted projective space $\mathbb{P}(a_0,\ldots, a_n)$ identifies with the GIT geometric quotient: 
\begin{equation*}
    \PP(a_0,\ldots,a_n) = (\mathbb{A}_{\Bbbk}^{n+1}\setminus \{0\})/\mathbb{G}_{m, \Bbbk}
\end{equation*}
of the open subset $(\mathbb{A}^{n+1}_{\Bbbk})^{ss}(\mathcal{O}_{\mathbb{A}_{\Bbbk}^{n+1}})=\mathbb{A}_{\Bbbk}^{n+1}\setminus \{0\}$ of semi-stable points of $\mathbb{A}_{\Bbbk}^{n+1}$ with respect to $\mathcal{O}_{\mathbb{A}_{\Bbbk}^{n+1}}$ endowed with the natural $\mathbb{G}_{m,\Bbbk}$-linearization. 

The homomorphism of graded $\Bbbk$-algebras: $$\Bbbk[x_0,\ldots,x_n]\to \Bbbk[z_0,\ldots,z_n], \quad x_i\mapsto z_i^{a_i},$$
where each $z_i$ has weight $1$, induces an isomorphism between $\Bbbk[x_0,\ldots,x_n]$ and the subring of invariants of $\Bbbk[z_0,\ldots,z_n]$ for the linear diagonal action of the product $\mu=\mu_{a_1}\times \cdots \mu_{a_n}$ of the finite $\Bbbk$-group schemes $\mu_{a_i}={\rm Spec} \, \Bbbk[t]/(t^{a_i}-1)$, $i=0,\ldots, n$. One deduces in turn that $\PP(a_0,\ldots,a_n)$ is isomorphic to the geometric quotient of $\PP(1,\ldots,1)=\PP^n$ by the induced diagonal action of $\mu$. It follows, among other properties, that $\PP(a_0,\ldots, a_n)$ is an irreducible normal projective variety, with at worse cyclic quotient singularities, and with class group $\mathrm{Cl}(\PP(a_0,\ldots, a_n))$ isomorphic to $\mathbb{Z}$, in particular $\PP(a_0,\ldots, a_n)$ is $\mathbb{Q}$-factorial.   

\begin{notation} When there is no risk of ambiguity, we will omit the subscript $\Bbbk$ in $\PP_{\Bbbk}(a_0, \cdots , a_n)$. On the other hand, to emphasize the choice of coordinates, we write 
\begin{equation*}
    \mathbb{P}(a_0, \ldots, a_n)_{[x_0:\cdots :x_n]}
\end{equation*}
for the weighted projective space with homogeneous coordinates $x_0,\ldots, x_n$ of weights $\mathrm{wt}(x_i)=a_i$.
\end{notation}

\medskip

We now collect additional basic notions and properties concerning weighted projective spaces.

\begin{lemma}[{\cite[$\S$ 5]{IF00}}]\label{lem:wps-reductions}\label{rem:well-formed}\label{def:well-formed}For a weighted projective space $\PP=\PP(a_0, \ldots, a_n)$, the following hold:

\begin{enumerate}

\item\label{lemma:gcd of weights} 
$\PP$ is isomorphic to $\PP(qa_0, \ldots, qa_n)$ for every positive integer $q$,.

\item\label{def:well-formed} 
$\PP$ is isomorphic to $\PP(b_0,\ldots,b_n)$ for suitable positive integers 
$(b_0,\ldots, b_n)$ such that $$\gcd \, (b_0, \ldots , b_{i-1}, \check{b_i}, b_{i+1}, \ldots , b_n) =1 \quad  \forall i=0,\ldots,n.$$

\end{enumerate}

\end{lemma}
\begin{proof} 
The first assertion follows from the facts that the identity homomorphism of $\Bbbk[x_0,\ldots,x_n]$ is an isomorphism of graded $\Bbbk$-algebras of degree $0$ between $\Bbbk(a_0,\ldots,a_n)$ and the $q$-th Veronese sub-algebra 
$$\Bbbk(qa_0,\ldots,qa_n)^{(q)}=\bigoplus_{m\geq 0} (\Bbbk(qa_0,\ldots,qa_n)^{(q)})_{m}=\bigoplus_{m\geq 0}\Bbbk(qa_0,\ldots,qa_n)_{mq}$$
of $\Bbbk(qa_0,\ldots,qa_n)$ and that the inclusion $\Bbbk(qa_0,\ldots,qa_n)^{(q)}\subset \Bbbk(qa_0,\ldots,qa_n)$ induces an isomorphism between the corresponding Proj's.

For the second assertion, taking assertion (1) into account we can assume without loss of generality that $\gcd(a_0,\ldots, a_n)=1$. For every $i=0,\ldots, n$, let
  \begin{equation*}
      d_i = \gcd(a_0,\ldots,a_{i-1},a_{i+1},\ldots,a_n) \quad \textrm{and} \quad  e_i = \mathrm{lcm}(d_0,\cdots d_{i-1},d_{i+1},\cdots d_n).
  \end{equation*}
Then $e_i$ divides $a_i$, $\gcd(e_i,d_i)=1$ and $e_id_i=\mathrm{lcm}(d_0,\ldots, d_n)=:e$. Moreover, the positive integers $b_i:=a_i/e_i$ satisfy $\gcd \, (b_0, \ldots , b_{i-1}, \check{b_i}, b_{i+1}, \ldots , b_n) =1$ for all $i=0,\ldots,n$. Let $\Bbbk[y_0,\ldots,y_n]\subset \Bbbk[x_0,\ldots,x_n]$ be the polynomial $\Bbbk$-subalgebra generated by the elements $y_i=x_i^{d_i}$, $i=0,\ldots, n$. Assigning the weight $eb_i=e_id_ib_i=a_id_i$ to $y_i$  identifies $\Bbbk[y_0,\ldots,y_n]=\Bbbk(eb_0,\ldots,eb_1)$ to the $e$-th Veronese sub-ring $\Bbbk(a_0,\ldots,a_n)^{(e)}$ of $\Bbbk(a_0,\ldots, a_n)$. Since on the other hand, we have: $\Bbbk(eb_0,\ldots,eb_1)^{(e)}=\Bbbk(b_0,\ldots,b_n)$, passing to the respective Proj's then gives the  following chain of isomorphisms: 
\[
\xymatrix{ \mathbb{P}(a_0,\ldots a_n) \ar[d]_{\cong} & &  \mathbb{P}(b_0,\ldots,b_n) \ar@{=}[d] \\ 
 \mathrm{Proj}(\Bbbk(a_0,\ldots,a_n)^{(e)})  
 \ar[r]^-{\cong} &  \mathbb{P}(eb_0,\ldots, eb_n) \ar[r]^-{\cong} &\mathrm{Proj}(\Bbbk(eb_0,\ldots,eb_n)^{(e)})  
}
\]
\end{proof}
Lemma \ref{lem:wps-reductions} (1) says in particular that one can always normalize the choice of weights in the description of a weighted projective space so that their greatest common divisor is equal to one. 
A description of a weighted projective $\PP(a_0,\ldots, a_n)$ in which the weights satisfy the condition
\[
  \gcd \, (a_0, \ldots , a_{i-1}, \check{a_i}, a_{i+1}, \ldots , a_n) =1 \quad (\forall \; i=0,\ldots, n)
  \]
is called {\it well-formed}. Lemma \ref{lem:wps-reductions} (2) says that every weighted projective space is isomorphic to a well-formed one.
There are several reasons to introduce the notion of being well-formed. For example, the usual adjunction formula for complete intersection in projective spaces extends for well-formed weighted completed intersections in (well-formed) weighted projective spaces, see Theorem \ref{theorem:adjuction formula} below. This also allows to give a purely arithmetic description of the singular loci of weighted projective spaces: 

\begin{proposition}[{\cite[5.15]{IF00}}]\label{rem:sing} The support of the singular locus $ \Sing\bigl(\PP(a_0,\ldots,a_n)\bigr)$ of a well-formed weighted projective space $\mathbb{P}(a_0,\ldots,a_n)$ is the union of the closed subsets 
    \begin{equation*}
        \Delta_I = \bigcap_{j\not\in I}\{x_j=0\}
    \end{equation*} where $I$ runs through the non-empty subsets $I\subset \{0,\ldots,n\}$ such that $\gcd\{a_i\mid i\in I\}>1$.
\end{proposition}

The {\it canonical sheaf} of a weighted projective space $\PP=\PP(a_0,\ldots, a_n)$ is the coherent reflexive sheaf of rank one defined to be $\omega_{\PP}:=i_* \omega_{\PP_{\mathrm{reg}}}$, where $\PP_{\mathrm{reg}}=\PP\setminus \mathrm{Sing}(\PP)$ is the regular locus of $\PP$, $i:\PP_{\mathrm{reg}}\hookrightarrow \PP$ is the open inclusion and  $\omega_{\PP_{\mathrm{reg}}}=\Lambda^n\Omega_{\PP_{\mathrm{reg}}/\Bbbk}$ is the canonical invertible sheaf of the regular scheme $\PP_{\mathrm{reg}}$. Any integral Weil divisor $K_{\PP}$ on $\PP$ such that $\omega_{\PP}\cong \mathcal{O}_{\PP}(K_{\PP})$ is called a canonical Weil divisor of $\PP$. For a well-formed expression $\PP=\PP(b_0,\ldots,b_n)_{[x_0:\ldots:x_n]}$ of $\PP$, the Weil divisor $-\sum_{i=0}^{n}H_{x_i}$, where $H_{x_i}=\{x_i=0\}\subset \PP$, is a canonical divisor $K_{\PP}$ and hence, we have an isomorphism of rank one reflexive sheaves $\omega_{\PP}\cong \mathcal{O}_\PP(-\sum_{i=0}^n b_i)$. 

\subsection{Cylindricity of weighted projective spaces}\label{2-2}
In this subsection, we observe the cylindricity of weighted projective spaces. 
\begin{proposition}\label{prop:adrien}
Let $\PP\coloneqq \PP(a_{0},\ldots,a_{n})_{[x_0:\cdots:x_n]}$ be a 
 weighted projective space and let $\{i_0,\ldots, i_m\}$ be a subset of indices such that $\gcd(a_{i_0},\ldots, a_{i_m})=1$. Then the affine open subset ${\mathbb D}_+(x_{i_0}\cdots x_{i_m})$ of $\mathbb{P}$ is isomorphic to $(\mathbb{A}_{\Bbbk}^1\setminus\{0\})^{m}\times \mathbb{A}_{\Bbbk}^{n-m}$.   
\end{proposition}
\begin{proof}
    Up to re-indexing, we can assume without loss of generality that $\{ i_0,\ldots,i_m \} =\{ 0,\ldots,m \}$. Putting $f:=x_0\cdots x_m$, we have ${\mathbb D}_+(f)\cong \mathrm{Spec}(A_0)$, where 
    $A_0$ denotes the subring of $$A:=
    \Bbbk\big{[} x_0,\ldots,x_n,f^{-1} \big{]} =\Bbbk \big{[} x_0^{\pm 1},\ldots, x_m^{\pm 1},x_{m+1},\ldots, x_n \big{]}$$ consisting of quasi-homogeneous elements of weighted degree $0$, that is, the ring of $\mathbb{G}_m$-invariants of $A$ for the induced $\mathbb{G}_m$-action. By assumption $\gcd(a_0,\ldots,a_m)=1$, there exists integers $b_0,\ldots, b_m$ such that $\sum_{i=0}^m a_ib_i=1$. Then $h:=\prod_{i=0}^m x_i^{b_i}\in A$ is a semi-invariant of weighted degree $1$ and for every $i=0,\ldots,n$, $y_i:=x_ih^{-a_i}$ an element of $A_0$. Moreover, the equality: 
    $$ A=\Bbbk \big{[} y_0^{\pm 1},\ldots, y_m^{\pm 1},y_{m+1}, \ldots, y_n \big{]} \big{[} h^{\pm 1} \big{]}$$ 
   implies that $A_0=\Bbbk[y_0^{\pm 1},\ldots, y_m^{\pm 1},y_{m+1}, \ldots, y_n]$. The elements $y_0,\ldots, y_m$ of $A_0$ satisfies the relation: $$\prod_{i=0}^{m}y_i^{b_i}=\prod_{i=0}^m x_i^{b_i}h^{-a_ib_i}=hh^{-1}=1.$$ Since $A_0$ is an integral $\Bbbk$-algebra of Krull dimension $n$, $A_0$ is thus equal to the quotient of the localization $\Bbbk[Y_0^{\pm 1},\ldots, Y_m^{\pm 1},Y_{m+1},\ldots, Y_n]$ of the polynomial ring in $n+1$ variables $Y_i$ by the principal ideal generated by $\prod_{i=0}^{m}Y_i^{b_i}-1$:
   \[
   A_0 
   \cong \Big{(} \Bbbk \big{[} Y_0^{\pm}, \cdots , Y_m^{\pm} \big{]} \big{/} \big{(} \prod_{i=0}^m Y_i^{b_i} -1 \big{)} \Big{)} \big{[} Y_{m+1}, \cdots , Y_n \big{]}
   \]
   Since $\gcd(b_0,\ldots, b_m)=1$, the ring $\Bbbk[Y_0^{\pm 1},\ldots, Y_m^{\pm 1}]/(\prod_{i=0}^{m}Y_i^{b_i}-1)$ is itself isomorphic to a Laurent polynomial ring in $m$ variables. In fact, as $\gcd(b_0,\ldots, b_m)=1$, there exist vectors:
   \[
  {\bf b}_j= (b_{j,0}, \cdots , b_{j,m}) \quad (1\leqq j \leqq m)
   \]
such that the matrix consisting of rows: 
\[
(b_0, \cdots, b_m) \, \,  {\rm and} \, \, {\bf b}_j \, \,  (1\leqq j \leqq m)
\]
yields an element of ${\rm GL}_{m+1}({\mathbb Z})$. Then the inverse of the automorphism:
\[
(Y_0, Y_1, \cdots , Y_{m}) \mapsto (Y_0^{b_0}\cdots Y_m^{b_m}, Y_0^{b_{1,0}}\cdots Y_m^{b_{1,m}}, \cdots, Y_0^{b_{m,0}}\cdots Y_m^{b_{m,m}} ) 
\]
of $\Bbbk$-algebra $\Bbbk [Y_0^{\pm}, \cdots , Y_m^{\pm}]$ sends $\prod_{i=0}^m Y_i^{b_i}$ to $Y_0$. Therefore it follows that $A_0$ is isomorphic to:
\[
\Bbbk[Y_1^{\pm}, \cdots , Y_m^{\pm}][Y_{m+1}, \cdots , Y_n], 
\]
which completes the proof. 
\end{proof}

\begin{corollary}\label{cor:wps-principal_cyl} Let $\mathbb{P}(a_0,\ldots, a_n)_{[x_0:\cdots:x_n]}$ be a weighted projective space. Then every affine open subset $\mathbb{D}_+(x_i)$, $i\in \{0,\ldots, n\}$,  is cylindrical.   
\end{corollary}
\begin{proof}
    Let $d=\gcd(a_0,\ldots, a_n)$ and $b_i=a_i/d$, $i=0,\ldots, n$. As in the proof of Lemma \ref{lem:wps-reductions}, let $d_i=\gcd(b_0,\ldots,b_{i-1},b_{i+1},\ldots, b_n)$, $e_i=\mathrm{lcm}(d_0,\ldots, d_{i-1},d_{i+1},\ldots,d_n)$, $c_i=b_i/e_i$ and $y_i=x_i^{d_i}$, $i=0,\ldots, n$. 
     Assigning the weight $ec_i$ to $y_i$ identifies $\Bbbk[y_0,\ldots, y_n]$ with the Veronese sub-algebra: 
     $\Bbbk(a_0,\ldots, a_n)^{(ed)}\subset \Bbbk(a_0,\ldots, a_n)$. The open subset $\mathbb{D}_+(x_i)=\mathbb{D}_+(x_i^{ed})$ of $\mathbb{P}(a_0,\ldots, a_n)$ then equals the inverse image of the open subset $\mathbb{D}_+(y_i)$ of $\mathrm{Proj}(\Bbbk[y_0,\ldots, y_n])$ by the associated isomorphism:
     $$f:\mathbb{P}(a_0,\ldots, a_n)\to \mathrm{Proj}\big{(} \Bbbk(a_0,\ldots, a_n)^{(ed)} \big{)}=\mathrm{Proj}(\Bbbk[y_0,\ldots, y_n]).$$ On the other hand, the $e$-th Veronese subring $\Bbbk[y_0,\ldots, y_n]^{(e)}$ of $\Bbbk[y_0,\ldots, y_n]$ endowed with the above grading equals $\Bbbk[y_0,\ldots,y_n]$ endowed with the grading defined by associating the weight $c_i$ to each $y_i$ and the corresponding isomorphism: $$h:\mathrm{Proj}(\Bbbk[y_0,\ldots, y_n])\to \mathrm{Proj}\big{(} \Bbbk[y_0,\ldots, y_n]^{(e)} \big{)}=\mathbb{P}(c_0,\ldots,c_n),$$ 
     identifies $\mathbb{D}_+(y_i)\subset \mathrm{Proj}(\Bbbk[y_0,\ldots, y_n])$ 
     to $\mathbb{D}_+(y_i)\subset \mathbb{P}(c_0,\ldots,c_n)$.  Since by construction $\mathbb{P}(c_0,\ldots,c_n)$ is now a well-formed weighted projective space, the conclusion that $\mathbb{D}_+(x_i)\cong \mathbb{D}_+(x_i^{de})\cong \mathbb{D}_+(y_i)$ is cylindrical follows from Proposition \ref{prop:adrien}.
\end{proof}
\begin{corollary} \label{theorem:A1-cylinder of WPS}\label{thm:cylinder_wps}\label{rem:adrien} For a weighted projective space $\PP\coloneqq \PP(a_{0},\ldots,a_{n})$, the following hold:

a) $\PP$ always contains an anti-canonically polar $\mathbb{A}^{1}$-cylinder (cf. \cite{KKW25}).

b) If $a_m=1$ for some $m\in \{0,\ldots, n\}$ then $\PP$ contains an anti-canonically polar $\mathbb{A}^n$-cylinder.
\end{corollary}

\begin{proof}
  The first assertion is immediate from Corollary \ref{cor:wps-principal_cyl} and the fact that weighted projective spaces are ${\mathbb Q}$-factorial varieties of divisor class  rank one. The second assertion follows from Proposition \ref{prop:adrien}. 
\end{proof}

\begin{remark}\label{exa:non-unipotent-wps}
Every weighted projective space $\mathbb{P}(1,a_1,\ldots,a_n)$ admits non-trivial actions of unipotent groups, in fact such a weighted projective space is even an equivariant completion of the vector group $\mathbb{G}_{a,\Bbbk}^n$, with respect to the triangular action $$[x_0:x_1:\cdots:x_n]\mapsto [x_0:x_1+t_1x_0^{a_1}:\cdots:x_n+t_nx_0^{a_n}]. $$ 
But even though they are always cylindrical by Corollary \ref{theorem:A1-cylinder of WPS}, 
not all well-formed weighted projective spaces admit non-trivial actions of unipotent groups. For instance, the weighted projective plane $\mathbb{P}(3,4,5)$ has automorphism group isomorphic to 
$\mathbb{G}_{m,\Bbbk}^3/\mathbb{G}_{m,\Bbbk}\cong \mathbb{G}_{m,\Bbbk}^2$, where $\mathbb{G}_{m,\Bbbk}\subset \mathbb{G}_{m,\Bbbk}^3$ 
is the diagonal subtorus $(\lambda^3,\lambda^4,\lambda^5)$, whence does not admit any non-trivial action of a unipotent group. 
\end{remark}
\subsection{Weighted complete intersections: well-formedness and quasi-smoothnes }\label{section3}
In this subsection, we review the concepts of well-formedness and quasi-smoothness for weighted complete intersections. For further details, the reader is referred to \cite{Dol81} and \cite{IF00}.
\begin{definition}\label{def:complete_intersection}
    {\rm Let $Y\subset \mathbb{P}(a_0, \ldots, a_n)_{[x_0:\cdots :x_n]}$ be a subvariety defined by a regular sequence of quasi-homogeneous polynomials:  
    \begin{equation*}
        f_1(x_0,\ldots,x_n)=\cdots = f_c(x_0,\ldots,x_n) = 0
    \end{equation*}
of degrees $d_1,\ldots,d_c$. We call $Y$ a} weighted complete intersection of multidegree {$(d_1,\ldots,d_c)$.}    
\end{definition}
\subsubsection{Well-formedness}\label{3-1}
It is expected that a subvariety $Y$ of a weighted projective space $\PP(a_0,\ldots,a_n)$ is singular along its intersection $Y$ with the singular locus of the ambient space. This motivates the following concept.
\begin{definition}[{\cite[Definition 6.9]{IF00}, \cite[Definition 2.6]{P23}}]\label{def:well-formed2}
    {\rm A subvariety $Y$ of a weighted projective space $\mathbb{P}$ is said to be} well-formed {\rm if $\mathbb{P}$ is well-formed and:
    \begin{equation*}
        \mathrm{codim}_Y \big{(} Y \cap \Sing(\mathbb{P}) \big{)} \geq 2.
    \end{equation*}
}

\end{definition}
A subvariety $Y$ of a weighted projective space which is not well-formed is not necessarily singular in codimension one, as illustrated by the following example:
\begin{example}\label{exa:well-formedness} By Proposition \ref{rem:sing}, the singular locus of the well-formed projective space $\PP=\mathbb{P}(1,1,2,2)_{[x:y:z:t]}$ is the curve $\{x=y=0\}\cong \PP(2,2)\cong \PP^1$. Since the divisor  $H_x=\{x=0\}$ contains $\mathrm{Sing}(\PP)$, it is not well-formed. On the other hand, $H_x\cong \PP(1,2,2)\cong \PP^2$ is regular.
\end{example}
The well-formedness of a hypersurface embedded in a weighted projective space can be checked by the following arithmetic conditions:
\begin{lemma}[{\cite[6.10]{IF00}}]\label{lem:well-formed}
    A weighted hypersurface $X \subset \PP(a_0,\ldots,a_n)$ of degree $d$ is well-formed if and only if the following conditions hold:
    \begin{enumerate}
        \item The weighted projective space $\mathbb{P}(a_0,\ldots,a_n)$ is well-formed.
        \item For all pairs of distinct indices $i,j$, $\gcd(a_0,\ldots,a_{i-1},a_{i+1},\ldots,a_{j-1},a_{j+1},\ldots,a_n)~|~d$.
    \end{enumerate}
\end{lemma}
The criterion of Lemma \ref{lem:well-formed} has the following generalization  to weighted complete intersections of higher codimension: 
\begin{lemma}[{\cite[6.12]{IF00}}]\label{lem:well-formed2}
    A weighted complete intersection $Y\subset \PP(a_0,\ldots,a_n)$ of multi-degree $(d_1, \ldots , d_c)$ 
    is well-formed if and only if the following two conditions hold:
    \begin{enumerate}
        \item The weighted projective space $\mathbb{P}(a_0,\ldots,a_n)$ is well-formed.
        \item For each $\mu \in \{1,\ldots,c\}$, the greatest common divisor of any collection of $(n-1-c+\mu)$ of the weights $a_i$ divides at least $\mu$ of the degrees $d_j$.
    \end{enumerate}
\end{lemma}

\subsubsection{Quasi-smoothness}\label{3-2}
Recall that there exists a natural quotient morphism: 
\begin{equation*}
    \pi\colon \mathbb{A}_{\Bbbk}^{n+1}\setminus\{0\} \longrightarrow \mathbb{P}(a_0,\ldots,a_n)=(\mathbb{A}_{\Bbbk}^{n+1}\setminus\{0\})/\mathbb{G}_{m,\Bbbk}.
\end{equation*}
\begin{definition}\label{def:quasi-smooth}
    {\rm For a closed subvariety $Y \subset \PP(a_0,\ldots,a_n)$, let $C_Y$ denote the closure in $\mathbb{A}^{n+1}_{\Bbbk}$ of the preimage $\pi^{-1}(Y) \subset \mathbb{A}_{\Bbbk}^{n+1} \setminus \{0\}$. We call $C_Y$ the} affine cone {\rm over $Y$. The subvariety $Y$ is said to be} quasi-smooth {\rm if its affine cone is smooth outside the origin.}
\end{definition}

In a similar way as a weighted projective space $\PP(a_0,\ldots,a_n)$ can be described as a quotient $\mathbb{P}^n/(\mu_{a_0}\times \cdots \times \mu_{a_n})$, a quasi-smooth subvariety $Y\subset \PP(a_0,\ldots, a_n)$ is locally isomorphic to the quotient of a non-singular variety by the action of a suitable finite $\Bbbk$-group subscheme of $\mathbb{G}_{m,\Bbbk}$, see e.g. \cite[Theorem 3.1.6]{Dol81}. 
This implies in particular that every quasi-smooth subvariety $Y\subset \PP(a_0,\ldots, a_n)$ is normal, $\mathbb{Q}$-factorial and has most cyclic quotient singularities.  

\medskip

Recall \cite[Definition 6.5]{IF00} that  a quasi-smooth hypersurface $X\subset \mathbb{P}(a_0,\ldots, a_n)_{[x_0:\ldots:x_n]}$ of degree  $d = a_i$ for some $i\in \{0,\ldots,n\}$ is called a {\it linear cone}. Up to scalar multiplication, such a hypersurface is the zero locus of a quasi-homogeneous polynomial of the form $x_i+F(x_0,\ldots, \hat{x}_i,\ldots, x_n)$ for some quasi-homogeneous polynomial $F$ of weighted degree $d$, hence is isomorphic to the weighted projective space  $\mathbb{P}(a_0,\ldots,a_{i-1},a_{i+1},\ldots,a_n)$.

%
The quasi-smoothness of a (general) hypersurface of a weighted projective space which is not a linear cone can be verified by using the following criterion:

\begin{lemma}[{\cite[Theorem 8.1]{IF00}}]\label{lem:quasi-smooth}
    Let $X \subset \mathbb{P}(a_0, \ldots, a_n)_{[x_0:\cdots :x_n]}$, $n\geq 1$, be a weighted hypersurface defined by a  quasihomogeneous polynomial $F_d(x_0,\ldots, x_n)$ of degree $d$ and  which is not a linear cone. 

  If $X$ is quasi-smooth, then for every nonempty subset $\mathcal{I} = \{i_1,\ldots,i_{k}\}$ of $\{0,\ldots,n\}$, one of the following holds:
    \begin{enumerate}
        \item $F_d$ contains a monomial $x_{i_1}^{m_1}\cdots x_{i_{k}}^{m_{k}}$ of degree $d$, or
        \item for each $\nu = 1,\ldots,k$, $F_d$ contains  a monomial $x_{i_1}^{m_{1,\nu}}\cdots x_{i_{k}}^{m_{k,\nu}}x_{e_{\nu}}$ of degree $d$, where the indices $\{e_{\nu}\}$ are $k$ distinct elements of $\{0,\ldots,n\}\setminus \mathcal{I}$. 
    \end{enumerate}

 Conversely, a weighted hypersurface  defined by a general quasihomogeneous polynomial $F_d$ of degree $d$ which satisfies one of the conditions (i) or (ii) above is quasi-smooth.  
\end{lemma}
The previous quasi-smoothness conditions for hypersurfaces have the following extension to general complete intersections of codimension two:
\begin{lemma}[{\cite[Theorem 8.7]{IF00}}]\label{lem:quasi-smooth2}
    Let $Y\subset \mathbb{P}(a_0, \ldots, a_n)_{[x_0:\cdots :x_n]}$ be a weighted complete intersection defined by two quasihomogeneous polynomials $f_1(x_0,\ldots,x_n)$ and $f_2(x_0,\ldots,x_n)$ of degrees $d_1$ and $d_2$ respectively and neither $f_1$ nor $f_2$ defines a linear cone. 
    

    If $Y$ is quasi-smooth then the following conditions are satisfied:
    \begin{enumerate}
        \item Single-variable condition. For each $i\in \{0,\ldots,n\}$, at least one of the following conditions holds:
    \begin{itemize}
        \item $f_1$ contains a monomial of the form $x_i^{p_i}$;
        \item $f_2$ contains a monomial of the form $x_i^{q_i}$;
        \item both $f_1$ and $f_2$ contain monomials of the form $x_i^{p_i}x_{e_1}$ and $x_i^{q_i}x_{e_2}$, respectively, with $e_1\neq e_2$.
    \end{itemize}

    \item Subset condition. For every nonempty subset $\mathcal{I} = \{i_1,\ldots,i_{k}\}\subset\{0,\ldots,n\}$ with $k\geq 2$, at least one of the following conditions holds:
    \begin{itemize}
        \item both $f_1$ and $f_2$ contain monomials in the variables $\{x_{i_1},\ldots,x_{i_{k}}\}$ of degrees $d_1$ and $d_2$, respectively;
        \item $f_1$ contains such a monomial, and $f_2$ contains monomials of the form $x_{i_1}^{q_{1,\mu}}\cdots x_{i_{k}}^{q_{k,\mu}}x_{e_\mu}$ for $k-1$ distinct indices $e_\mu\notin\mathcal{I}$;
        \item $f_2$ contains such a monomial, and $f_1$ contains monomials of the form $x_{i_1}^{p_{1,\mu}}\cdots x_{i_{k}}^{p_{k,\mu}}x_{e_\mu}$ for $k-1$ distinct indices $e_\mu\notin\mathcal{I}$;
        \item for each $\mu=1,\ldots,k$, there exist monomials of the forms $x_{i_1}^{p_{1,\mu}}\cdots x_{i_{k}}^{p_{k,\mu}}x_{e_{1,\mu}}$ in $f_1$ and $x_{i_1}^{q_{1,\mu}}\cdots x_{i_{k}}^{q_{k,\mu}}x_{e_{2,\mu}}$ in $f_2$, where the sets $\{e_{1,\mu}\}$ and $\{e_{2,\mu}\}$ consist of distinct indices, and together they contain at least $k+1$ distinct indices outside $\mathcal{I}$.
    \end{itemize}
    \end{enumerate}
    
    Conversely, a general weighted complete intersection which satisfies the conditions above is quasi-smooth.
    
\end{lemma}


The following result asserts that singularities of quasi-smooth well-formed weighted complete intersections arise from those of ambient weighted projective space. 
\begin{proposition}[{\cite[Proposition 8]{Di86}}]
    Let $Y \subset \PP(a_0,\ldots,a_n)$ be a quasi-smooth, well-formed weighted complete intersection. Then its singular locus coincides with the intersection of $Y$ with the singular locus of $\PP(a_0,\ldots,a_n)$.
\end{proposition}
The dualizing sheaf $\omega_Y$ of quasi-smooth weighted complete intersection is the coherent reflexive sheaf of rank one $    \omega_Y\cong i_*\omega_{Y_{\mathrm{reg}}}$, where $i\colon Y_{\mathrm{ref}} \hookrightarrow Y$ is the inclusion of the regular locus $Y_{\mathrm{reg}}$ of $Y$ and  $\omega_{Y_{\mathrm{reg}}}=\bigwedge^{\dim Y}\Omega_{Y_{\mathrm{reg}}/\Bbbk}$ is the canonical invertible sheaf of $Y_{\mathrm{reg}}$. Any integral Weil divisor $K_Y$ on $Y$ such that $\omega_Y\cong \mathcal{O}_Y(K_Y)$ is called a canonical divisor of $Y$. One of the significances and utilities of well-formed, quasi-smooth weighted complete intersections amounts to the following fact, so-called the adjunction formula:   
\begin{theorem}[{\cite[Theorem 3.3.4]{Dol81}}]\label{theorem:adjuction formula}
    Let $Y \subset \PP=\mathbb{P}(a_0,\ldots,a_n)$ be a quasi-smooth, well-formed weighted complete intersection of multidegree $(d_1,\ldots, d_c)$. Then: 
    \begin{equation*}
        \omega_Y\cong \bigg(\omega_{\PP}\otimes \mathcal{O}_{\PP}\Big(\sum_{i=1}^c d_i\Big)\bigg)\bigg|_Y\cong \mathcal{O}_{\PP}\Big(\sum_{i=1}^c d_i - \sum_{j=0}^n a_j\Big)\Big|_Y.
    \end{equation*}
\end{theorem}

\section{Cylindricity of weighted hypersurfaces}\label{section4}

Theorem \ref{theorem:adjuction formula} implies in particular that the canonical divisor of a well-formed quasi-smooth hypersurface of degree $d$ of a weighted projective space $\mathbb{P}(a_0,\ldots, a_n)$ is either ample if $d>\sum a_i$, or trivial if $d=\sum a_i$ or anti-ample otherwise. Since such a weighted hypersurface has at worst cyclic quotient singularities which guarantees the ${\mathbb Q}$-factoriality, 
Corollary \ref{cor:Kpseff+lc_implies_nocylinder} implies that a cylindrical weighted hypersurface $X$ of $\mathbb{P}(a_0,\ldots, a_n)$ is a Fano variety. In this section, we review results concerning the cylindricity of weighted Fano hypersurfaces in weighted projective spaces.


\subsection{A class of cylindrical Fano weighted hypersurfaces}\label{4-1}
Every (well-formed) quasi-smooth hypersurface $X\subset \mathbb{P}(a_0,\ldots, a_n)$, $n\geq 2$, of degree $d=a_i$ for some $i$, is a linear cone, isomorphic to the weighted projective space $\mathbb{P}(a_0,\ldots, \hat{a}_i,\ldots, a_n)$, whence, by Corollary \ref{thm:cylinder_wps}, is a cylindrical Fano variety. The next result provides another class of cylindrical Fano hypersurfaces: 


\begin{theorem} \label{thm:daiaj-cyl} A well-formed quasi-smooth hypersurface $X\subset \mathbb{P}(a_0,\ldots, a_n)$, $n\geq 3$, defined over a quadratically closed field $\Bbbk$ and of degree $d=a_i+a_j$ for some  $i\neq j$  contains an anti-canonically polar $\mathbb{A}^1$-cylinder.
\end{theorem}
\begin{proof}
The hypersurface $X$ is given by the vanishing of a quasi-homogeneous polynomial $F\in \Bbbk[z_0,\ldots, z_n]$ of weighted degree $d$ with respect to a system of variables $z_0,\ldots,z_n$ with respective weights $a_i$.  We claim that there exists a coordinate system $x_0,\ldots,x_n$ on $\Bbbk[z_0,\ldots,z_n]$ of respective weights $a_i$ and an isomorphism of graded $\Bbbk$-algebras $\Bbbk[x_0,\ldots, x_n]\cong \Bbbk[z_0,\ldots,z_n]$ of degree $0$ such that in the coordinate system $x_0,\ldots,x_n$, 
 $F$ has the form:  
\begin{equation}\label{eq:normal_form}
   F=x_{n-1}x_n+G(x_0,\ldots,x_{n-2})
\end{equation}
for some quasi-homogeneous polynomial $G$ of weighted degree $d$. This can be seen as follows: 

Without loss of generality, we can assume that $i=n-1$ and $j=n$.  We show first that $F$ contains either $z_{n-1}z_{n}$ or a term of the form $\lambda z_{n-1}^{2}+\mu z_{n}^{2}$ with $\lambda\mu\neq0$, or a monomial of the form $z_{i}z_{n-1}$ or $z_{i}z_{n}$ for some $i\in\{0,\ldots,n-2\}$. Write: $$F=h_{d}(z_{n-1},z_{n})+\sum_{i=0}^{n-2}z_{i}h_{d-a_{i}}(z_{n-1},z_{n})+H,$$ where $h_{e}$ is a quasi-homogeneous polynomial of weighted degree $e$ and $H$ is a quasi-homogeneous polynomial of weighted degree $d$ with no terms of degree $\leq1$ in the variables $z_{i}$, $i=0,\ldots,n-2$. It is straightforward to verify that for every two indices $i,j\in\{0,\ldots n-2\}$, possibly equal, for which there exist $p_{i},q_{j}\geq1$ such that $p_{i}a_{n-1}+a_{i}=q_{j}a_{n}+a_{j}=d=a_{n-1}+a_{n}$, we have either $p_{i}=1$ and $a_{n}=a_{i}$ or $q_{j}=1$ and $a_{n-1}=a_{j}$.
If one of the polynomials $h_{d-a_{i}}$, $i\in\{0,\ldots,n-2\}$ contains $z_{n-1}$ or $z_{n}$ among its monomials, then we are done. Otherwise, the previous arithmetic observation implies that, up to exchanging $z_{n-1}$ and 
$z_{n}$, we have, for all $i\in\{0,\ldots,n-2\}$, $h_{d-a_{i}}=\alpha_{i}z_{n-1}^{p_{i}}$ for some $\alpha_{i}\in \Bbbk$ and $p_{i}\geq2$. The restriction of the gradient of $F$ to the plane $\Pi=\{z_{0}=\cdots z_{n-2}=0\}\subset\mathbb{A}^{n+1}_{\Bbbk}$ then has the form $$\nabla F|_{\Pi}=\left((\alpha_{\ell}z_{n-1}^{p_{\ell}})_{\ell=0,\ldots,n-2},\frac{\partial h_{d}}{\partial z_{n-1}},\frac{\partial h_{d}}{z_{n}}\right).$$ 
It follows that $h_{d}$ contains a monomial of degree $\leq1$ in $z_{n-1}$, for otherwise $C_{X}$ would be singular along the coordinate line $\{z_{0}=\cdots=z_{n-1}=0\}$. So either $h_{d}$ contains the monomial $z_{n-1}z_{n}$ and we are done, or it is of the form $\lambda z_{n-1}^{p}+\mu z_{n}^{q}$, with $p,q\geq2$ and $\mu\neq0$. If $\lambda=0$ then, by Lemma \ref{lem:quasi-smooth} (1) applied to $I=\{n-1\}$, we have $\alpha_{i}\neq0$ for some $i\in\{0,\ldots,n-2\}$ and then, $a_{n}+a_{n-1}=qa_{n}=p_{i}a_{n-1}+a_{i}$ which is impossible since $p_{i},q\geq2$ and $a_{i},a_{n}\geq1$. 
Thus, $\lambda\neq0$ and since $a_{n-1}+a_{n}=d$, the only possibility is $p=q=2$ and we are done. 

\medskip 

So up to re-indexing the variables $z_0,\ldots, z_n$, we can now assume without loss of generality that $F$ contains either a monomial $z_{n-1}z_{n}$ or, up to a scalar multiplication, a term of the form $z_{n-1}^{2}+\mu z_{n}^{2}$ with $\mu\neq0$. The second case can occur if and only if $a_{n-1}=a_n$ and since, by assumption, $\Bbbk$ is quadratically closed, it then reduces to the first one after the homogeneous coordinate change $(z_{n-1},z_{n})\mapsto(z_{n-1}-\nu z_{n},z_{n-1}+\nu z_{n})$ of degree $0$ for some square root $\nu$ of $-\mu$ in $\Bbbk^*$. We can therefore assume that $F$ contains $z_{n-1}z_n$ and write $$F=z_{n}(z_{n-1}+F_{1}(z_{0},\ldots,z_{n-2}))+F_{2}(z_{0},\ldots,z_{n-1}),$$ where $F_{1}$ and $F_{2}$ are quasi-homogeneous polynomials of weighted degrees $a_{n-1}$ and $d$ respectively. After making the homogeneous coordinate change $x_{n-1}=z_{n-1}+F_{1}(z_{0},\ldots,z_{n-2})$ of degree $0$, we have: $$F=x_{n-1}(z_{n}+F_{3}(z_{0},\ldots,z_{n-2}, x_{n-1}))+F_{4}(z_{0},\ldots,z_{n-2})$$ and then, making the homogeneous coordinate change $x_{n}=z_{n}+F_{3}(x_{0},\ldots,z_{n-2}x_{n-1})$ and $x_i=z_i$ for $i\in\{0,\ldots, n-2\}$ of degree $0$, brings $F$ into the desired form \eqref{eq:normal_form}.

To conclude the proof, we observe that the rational projection: $$\pi:\mathbb{P}(a_0,\ldots, a_n)_{[x_0:\cdots x_n]}\dashrightarrow \mathbb{P}(a_0,\ldots,a_{n-1})_{[y_0:\ldots, y_{n-1}]}, \quad [x_0:\ldots:x_n]\mapsto [x_0:\ldots :x_{n-1}]$$ induces a birational map $X\dashrightarrow \mathbb{P}(a_0,\ldots,a_{n-1})$  restricting to an isomorphism: $$X\cap \mathbb{D}_+(x_{n-1})\cong \mathbb{D}_+(y_{n-1}) \subset \mathbb{P}(a_0,\ldots,a_{n-1}).$$ 
Corollary \ref{cor:wps-principal_cyl} and its proof imply in turn that $\mathbb{D}_+(y_{n-1})$ contains a cylinder $U$ whose complement is the support of an effective $\mathbb{Q}$-divisor $D\sim_{\mathbb{Q}} mH$ for some $m$, where $H$ is any ample Weil $\mathbb{Q}$-divisor on $\PP(a_0,\ldots,a_{n-1})$. 
Together with the adjunction formula of Theorem \ref{theorem:adjuction formula}, this implies that $U\subset\mathbb{D}_+(y_{n-1} )\cong X\cap \mathbb{D}_+(x_{n-1})\subset X$ is  an anti-canonical polar cylinder.
\end{proof}

\begin{remark} In Theorem \ref{thm:daiaj-cyl}, the hypothesis that the base field $\Bbbk$ is quadratically closed cannot be relaxed in general. For instance, the real anisotropic smooth quadric surface: $$Q_2=\{x_0^2+x_1^2+x_2^2+x_3^2=0\}\subset \mathbb{P}^3_{\mathbb{R}}$$ is birationally ruled but not cylindrical. Indeed $Q_2$ is isomorphic to the product $Q_1\times_{\mathbb{R}} Q_1$, where $Q_1\subset \mathbb{P}^2_{\mathbb{R}}$ is the anisotropic real conic $\{x_0^2+x_1^2+x_2^2=0\}$. Since the fiber of the projection $\mathrm{pr}_2:Q_1\times_{\mathbb{R}} Q_1\to Q_1$ over the generic point of $Q_1$ has a $\Bbbk(Q_1)$-rational point, it follows that  $Q_1\times_{\mathbb{R}}Q_1$ is birationally isomorphic over $Q_1$ to the total space of the trivial $\mathbb{P}^1$-bundle 
 $\mathbb{P}^1_{\mathbb{R}}\times_{\mathbb{R}} Q_1\to Q_1$. 
 On the other hand, since $Q_2(\mathbb{R})=\emptyset$, the existence of a cylinder in $Q_1\times_{\mathbb{R}}Q_1$ would imply that it is  isomorphic over $Q_1$ to the total space of the trivial $\mathbb{P}^1$-bundle $\mathbb{P}^1_{\mathbb{R}}\times_{\mathbb{R}} Q_1\to Q_1$, which is impossible since the latter does not admit any closed embedding as quadric surface in $\mathbb{P}^3_{\mathbb{R}}$. 

To our knowledge the question of existence of cylindrical smooth real isotropic quadrics $Q_{n-1}=\{\sum_{i=0}^nx_i^2=0\}\subset \mathbb{P}^n_{\mathbb{R}}$
in higher dimension is an open problem. In this direction, it follows from \cite{Kar03} 
and \cite[Theorem 6.4]{Kar00}
 that if $n=2^s$ for some $s\geq 2$ then $Q_{n-1}$ is not birationally ruled, whence not cylindrical. In particular, the smooth real isotropic quadric threefold $Q_3$ is not cylindrical.   
\end{remark}

\begin{corollary}
    \label{thm:H-polar}
A well-formed quasi-smooth hypersurface $X\subset \mathbb{P}(a_0,\ldots, a_n)$, $n\geq 4$, defined over a quadratically closed field $\Bbbk$ and of degree $d=a_i+a_j$ 
for some  $i\neq j$  
contains an $H$-polar cylinder for every ample divisor $H$.

\end{corollary}
\begin{proof} Since $X$ is $\mathbb{Q}$-factorial and $\dim X\geq 3$, the conclusion follows from \cite[Theorem 3.2.4]{Di86} which asserts that $\mathrm{Pic}(X)\cong \mathbb{Z}$. 
\end{proof}
\subsection{Non-Cylindrical Fano Hypersurfaces}\label{4-2}

\subsubsection{Examples of non-cylindrical del Pezzo hypersurfaces}
\label{subsec:NonCyl-dP}
In this subsection we
review a selection of known results concerning the anti-canonical polar cylindricity of del Pezzo hypersurfaces defined over an algebraically closed field $\Bbbk$ of characteristic zero, in connection to Conjecture \ref{Conj:CPPZ} relating the non-existence of anti-canonical polar cylinders to $K$-polystability. 

\medskip
Quasi-smooth, well-formed del Pezzo hypersurfaces $S\subset \mathbb{P}(a_0,a_1,a_2,a_3)$ are classified in terms of the tuple of integers $\underline{a}=(a_0,a_1,a_2,a_3)$ and their degree $d$, see \cite{Pae16}. The classification can be divided into thirty five infinite families of pairs $(\underline{a},d)$   
depending on a positive integer parameter \cite[Table 7]{Pae16}, reproduced in Table \ref{tab:hogehoge} and a finite number of so-called sporadic cases. The non-existence of anti-canonical polar cylinders for several of these surfaces has been verified as a consequence of Theorem \ref{theorem:alpha cylinder} and of the classification of such hypersurfaces with $\alpha$-invariant greater than or equal to $1$, see e.g., \cite{JK01,CPS10}. However, it turns to exist infinitely many families of del Pezzo hypersurfaces $S$ such that $\alpha(S) < 1$ and for these, the question of existence or non-existence of anti-canonical polar cylinder requires a case-by-case analysis using  list in \cite{Pae16}. 
Several sporadic case of this type, including for instance smooth del Pezzo surfaces of degree $3$ and $2$, have also been verified not to have anti-canonical polar cylinders by means of more advanced methods, see \cite{CPW0,CPW}. On the other hand, it was recently established that all but except possibly finitely many surfaces classes of surfaces belonging to the thirty five infinite families do not admit an anti-canonical polar cylinder, namely:  
\begin{theorem}[{\cite[Theorem 4.7]{KKW24}}]\label{thm:KKW24}
    No quasi-smooth member in the thirty five infinite families  in Table  \ref{tab:hogehoge}, contains an anti-canonical polar cylinder for $n > 2$.
\end{theorem}

\begin{remark} In Table  \ref{tab:hogehoge}, the infinite series \textnumero~23 - 35 which all have bounded index $1\leq I\leq 6$ were shown earlier in \cite{CPS10} to have $\alpha$-invariant equal to $1$, so the fact that they do not admit any anti-canonical polar cylinder follows from Theorem \ref{theorem:alpha cylinder}.
\end{remark}

\begin{remark}    
The last three columns in Table  \ref{tab:hogehoge} display what is known so far concerning the $K$-stability or $K$-unstability of the surfaces in the thirty five infinite families, see \cite{KKW24} for the details . It shows in particular that for $n\geq 4$, the surfaces \textnumero~1 to 22 
are $K$-unstable and do not contain anti-canonical polar cylinders, whence are counter-examples to Conjecture \ref{Conj:CPPZ}.
\end{remark}

\begin{table}[htbp]
    \centering
    \caption{Quasi-smooth, well-formed del Pezzo hypersurfaces: Infinite Series}
    \label{tab:hogehoge}
    \rowcolors{2}{gray!15}{white}
    \begin{tabular}{@{}c|l|c|c|c|c@{} }
        \toprule
        \textnumero & $(a_0,a_1,a_2,a_3)$ & degree & $K$-stable & $K$-unstable & Unknown\\
        \midrule
            1 & $(1,3n-2,4n-3,6n-5)$ & $12n-9$ & $n=1, 2$ & $n\geq 3$  &\\
            2 & $(1,3n-2,4n-3,6n-4)$ & $12n-8$ & $n=1, 2$ & $n\geq 3$  &\\
            3 & $(1,4n-3,6n-5,9n-7)$ & $18n-14$ & $n=1, 2$ & $n\geq 3$  &\\
            4 & $(1,6n-5,10n-8,15n-12)$ & $30n-24$ & $n=1, 2$ & $n\geq 3$  &\\
            5 & $(1,6n-4,10n-7,15n-10)$ & $30n-20$ & $n=1, 2$ & $n\geq 3$  &\\
            6 & $(1,6n-3,10n-5,15n-8)$ & $30n-15$ & $n=1$ & $n\geq 2$  &\\
            7 & $(1,8n-2,12n-3,18n-5)$ & $36n-9$ &  & $n\geq 1$  &\\
            8 & $(2,6n-3,8n-4,12n-7)$ & $24n-12$ & $n=1$ & $n\geq 4$ & $n=2, 3$\\
            9 & $(2,6n+1,8n+2,12n+3)$ & $24n+6$ & & $n\geq 3$ & $n=1, 2$\\
            10 & $(3,6n+1,6n+2,9n+3)$ & $18n+6$ & & $n\geq 3$ & $n=1, 2$\\
            11 & $(7,28n-18,42n-27,63n-44)$ & $126n-81$ & & $n\geq 4$ & $n=1, 2, 3$\\
            12 & $(7,28n-17,42n-29,63n-40)$ & $126n-80$ & $n=1$ & $n\geq 3$ & $n=2$\\
            13 & $(7,28n-13,42n-23,63n-31)$ & $126n-62$ & & $n\geq 3$ & $n=1, 2$\\
            14 & $(7,28n-10,42n-15,63n-26)$ & $126n-45$ & & $n\geq 3$ & $n=1, 2$\\
            15 & $(7,28n-9,42n-17,63n-22)$ & $126n-44$ & & $n\geq 3$ & $n=1, 2$\\
            16 & $(7,28n-6,42n-9,63n-17)$ & $126n-27$ & & $n\geq 3$ & $n=1, 2$\\
            17 & $(7,28n-5,42n-11,63n-13)$ & $126n-26$ & & $n\geq 3$ & $n=1, 2$\\
            18 & $(7,28n-2,42n-3,63n-8)$ & $126n-9$ & & $n\geq 3$ & $n=1, 2$\\
            19 & $(7,28n-1,42n-5,63n-4)$ & $126n-8$ & & $n\geq 3$ & $n=1, 2$\\
            20 & $(7,28n+2,42n+3,63n+1)$ & $126n+9$ & & $n\geq 3$ & $n=1, 2$\\
            21 & $(7,28n+3,42n+1,63n+5)$ & $126n+10$ & & $n\geq 3$ & $n=1, 2$\\
            22 & $(7,28n+6,42n+9,63n+10)$ & $126n+27$ & & $n\geq 3$ & $n=1, 2$\\
            23 & $(2,2n+1,2n+1,4n+1)$ & $8n+4$ & $n\geq 1$ & &\\
            24 & $(3,3n,3n+1,3n+1)$ & $9n+3$ & $n\geq 1$ & &\\
            25 & $(3,3n+1,3n+2,3n+2)$ & $9n+6$ & $n\geq 1$ & &\\
            26 & $(3,3n+1,3n+2,6n+1)$ & $12n+5$ & $n\geq 1$ & &\\
            27 & $(3,3n+1,6n+1,9n)$ & $18n+3$ & $n\geq 1$ & &\\
            28 & $(3,3n+1,6n+1,9n+3)$ & $18n+6$ & $n\geq 1$ & &\\
            29 & $(4,2n+3,2n+3,4n+4)$ & $8n+12$ & $n\geq 1$ & &\\
            30 & $(4,2n+3,4n+6,6n+7)$ & $12n+18$ & $n\geq 1$ & &\\
            31 & $(6,6n+3,6n+5,6n+5)$ & $18n+15$ & $n\geq 1$ & &\\
            32 & $(6,6n+5,12n+8,18n+9)$ & $36n+24$ & $n\geq 1$ & &\\
            33 & $(6,6n+5,12n+8,18n+15)$ & $36n+30$ & $n\geq 1$ & &\\
            34 & $(8,4n+5,4n+7,4n+9)$ & $12n+23$ & $n\geq 1$ & &\\
            35 & $(9,3n+8,3n+11,6n+13)$ & $12n+35$ & $n\geq 1$ & &\\
        \bottomrule
        \rowcolor{white}
        \multicolumn{3}{l}{\footnotesize $n$ denotes a positive integer.}
    \end{tabular}
\end{table}

\medskip

We now give an illustration of the methods employed in \cite{KKW24} to establish Theorem \ref{thm:KKW24} by considering the case of family \textnumero~4, that is, quasi-smooth well-formed del Pezzo hypersurfaces: $$S:=S_{30n-24}\subset \mathbb{P}(1,6n-5,10n-8,15n-12)_{[x:y:z:t]}$$ of degree $d=30n-24$ and index $I=n$. For $n=1$, it is simply a smooth del Pezzo surface of degree $6$ in $\mathbb{P}(1,1,2,3)$. We henceforth assume that $n\geq 2$. Then up to a coordinate change, $S$ is defined by a quasi-homogeneous polynomial of the form: 
\begin{equation}
\label{eq:S_30n-24}
F_n(x,y,z,t)=t^2+z^3+zx^{2n-1}f_{18n-15}(x,y)+x \big{(} y^5+g_{30n-25}(x,y)+x^{30n-25} \big{)},\end{equation}
where $f_{18n-25}(x,y)$ and $g_{30n-25}(x,y)$ are appropriate quasi-homogeneous polynomials such that $g_{30n-25}$ contains no monomials in $y^5$, $y^4x^{6n-5}$ and $x^{30n-25}$. 

We denote by $C_x=\{x=t^2+z^3=0\}\subset S$ the intersection of $S$ with the hyperplane $H_x=\{x=0\}\subset \mathbb{P}(1,6n-5,10n-8,15n-12)$. By the quasi-smoothness assumption, $S\setminus C_x=S\cap \mathbb{D}_+(x)$ is a smooth $\Bbbk$-scheme. On the other hand, $S$ has singularities supported on $C_x$, which are cyclic quotient singularities of type $\frac {1}{r}(1,a)$ where $r$ divides either $6n-5$, $10n-8$ or $15n-12$. Moreover, it is straightforward to verify from the Jacobian criterion that if $S=Z(F_n)$ is quasi-smooth then the surface $Z(F_1)$ in $\mathbb{P}(1,1,2,3)$ is a smooth del Pezzo surface of degree $1$ for which we have an isomorphism: 
\begin{equation}
    \label{eq:dP6-open}
S\setminus C_x=Z(F_n)\cap \mathbb{D}_+(x)\cong Z(F_1)\cap \mathbb{D}_+(x)=V \big{(} F_1(1,y,z,t) \big{)}\subset \mathbb{A}^3_{\Bbbk}
\end{equation}


It is readily verified by a weighted blow-up computation that the log-canonical threshold of the log pair $(S,C_x)$ is at most $5/6$. Since $-K_S\sim nC_x$ it follows that: 
$$\alpha(S)\leq \mathrm{lct}(S,nC_x)\leq \frac{5}{6n}<\frac{1}{3},$$
whence, by inequality \eqref{eq:alpha_delta} that $\delta(S)< 1$ 
which implies, by Theorem \ref{thm:BJ}, that $S$ is not $K$-semistable. We now turn to the non-existence of anti-canonical polar cylinders in $S$.

\begin{proposition}  
\label{theorem:no cylinder surface}
   A quasi-smooth, well-formed hypersurface $S\subset \mathbb{P}(1, 6n-5, 10n-8, 15n-12)$ of degree $d=30n-24$, where $n\geq 1$,  does not contain an anti-canonical polar cylinder. 
   \end{proposition}

\begin{proof}
The case $n=1$, where $S$ is a smooth del Pezzo surface of degree $1$, is done by \cite[Proposition 5.1]{KPZ0}. We henceforth assume that $n\geq 2$ in what follows. 

Recall that for a $\mathbb{Q}$-Cartier divisor $D$ on a surface $X$ with cyclic quotient singularities and a singular point $\msp \in X$ of type $\frac{1}{r}(1,a)$, where $r$ and $a$ are coprime, the multiplicity $\mult_{\msp}(D)$ of $D$ at $\msp$ is defined as the multiplicity
$\mult_{\widetilde{\msp}}(\widetilde{D})$, where 
    $\phi\colon \widetilde{\mathcal{U}}  \longrightarrow \mathcal{U}$ 
 is a cyclic cover of degree $r$ of an open neighborhood $\mathcal{U}$ branched only over $\msp$ and with smooth total space  $\widetilde{\mathcal{U}}$, $\widetilde{\msp}$ is the unique point over $\msp$ and $\widetilde{D} = \phi^{-1}(D|_{\mathcal{U}})$. Recall further \cite[Lemma 8.12]{Kol97} that if $D$ is effective, the log pair $(\mathcal{U},D|_{\mathcal{U}})$ is log-canonical at $\msp$ if and only if the log pair $(\tilde{\mathcal{U}},\widetilde{D})$ is log canonical at $\tilde{\msp}$ and that by \cite[Proposition 9.5.13]{Laz04}, this holds if and only if $\mult_{\msp}(D)=\mult_{\widetilde{\msp}}(\widetilde{D})\leq 1$. 
 
Now assume that $D\sim_{\mathbb{Q}} -K_S\sim_{\mathbb{Q}} nC_x$ is an effective anti-canonical $\mathbb{Q}$-divisor. For any point $\msp\in S\setminus C_x$, there exists a member of the complete linear system: $$\big{|} (10n-8)H_x \big{|}= \big{\{} \,  \{ az + byx^{4n-3} + cx^{10n-8} = 0\} \, \big{|}  \quad [a:b:c]\in \mathbb{P}^2 \, \big{\}}$$ 
  on the ambient space $\mathbb{P}(1, 6n-5, 10n-8, 15n-12)$ whose intersection $M\sim_{\mathbb{Q}} (10n-8)C_x$ with $S$ passes through $\msp$ and is not contained in the support of $D$. Computing the intersection number gives the inequality: 
\begin{equation*}
    \frac{2n}{6n-5} = \frac{n(10n-8)(30n-24)}{(6n-5)(10n-8)(15n-12)} = D\cdot M \geq \mult_{\msp}(D).
\end{equation*}
Since $n\geq 2$ by assumption, it follows that $1\geq \frac{2n}{6n-5}$, which implies the log pair $(S,D)$ is  log-canonical at every point $\msp\in S\setminus C_x$. Similarly, if the support of $D$ does not contain $C_x$, then for every point $\msp\in C_x$,  we obtain the following estimate: 
\begin{equation*}
    \frac{n}{(6n-5)(5n-4)} = \frac{n(30n-24)}{(6n-5)(10n-8)(15n-12)}= D\cdot C_x \geq \frac{\mult_{\msp}(D)}{r}
\end{equation*}
if $\msp$ is a singular point of type $\frac{1}{r}(1,a)$. Since $n\geq 2$ and $r$ divides either $6n-5$ or $10n-8$ or $15n-12$, it follows that $1\geq \mult_{\msp}(D)$ for every point $\msp\in C_x$, which implies that the log pair $(S,D)$ is log canonical at every point of $S$. In conclusion, if $D$ is an effective anti-canonical $\mathbb{Q}$-divisor such that the log pair $(S,D)$ is not log canonical at a point $\msp$ of $S$,  then $\msp\in C_x$ and  $C_x$ is contained in the support of $D$. 

Now assume that $S$ contains a cylinder $U=S\setminus\mathrm{Supp}(D)\cong Z\times_{\Bbbk} \mathbb{A}^1_{\Bbbk}$ for some effective anti-canonical $\mathbb{Q}$-divisor $D=\sum d_i D_i$. By Corollary  \ref{cor:cyl_implies_notlc}, the log pair $(S,D)$ is not log-canonical at some point $\msp$ of $S$. The previous observation then implies that  $\msp\in C_x$ and that $C_x$ is contained in 
the support of $D$, say $D_1=C_x$ with $d_1\leq n$ as $D\sim_{\mathbb{Q}} nC_x$. Note that since $S\setminus C_x$ is isomorphic to the complement of an irreducible effective anti-canonical divisor in a smooth del Pezzo surface of degree $6$ by \eqref{eq:dP6-open}, its Picard rank is equal to $8$, and hence, it cannot be a cylinder over a smooth rational affine curve. Thus, $d_1<n$ and $D=d_1C_x+R$ where $R$ is a non-trivial effective $\mathbb{Q}$-divisor whose support does not contain $C_x$.

Let $\varphi:S\dashrightarrow  \mathbb{P}^1_{\Bbbk}$ induced by the projection $\mathrm{pr}_Z:U\to Z$. Assume that $\varphi$ is a morphism. Then by \eqref{eq:ram_2} in the proof of Proposition \ref{prop:pseff_cylinder}, precisely one of the irreducible components $D_i$ of $D$ is a section of $\varphi$ and $d_i=2$, implying that the log pair $(S,D)$ is not log-canonical at every point of $D_i$. So it must be that $i=1$ and $d_1=2$. So $n>2$ and then,  $D':=\frac{n}{n-2}D-\frac{2n}{n-2}C_x$
would be an effective anti-canonical $\mathbb{Q}$-divisor on $S$ intersecting trivially the general fibers of $\varphi:S\to \mathbb{P}^1$, which is absurd.

So $\varphi$ has a unique proper base point $\msp$, supported on $C_x$ and at which the log pair $(S,D)$ is not log-canonical. But then, since $d_1<n$, $D'=\frac{n}{n-d_1}D-\frac{d_1n}{n-d_1}C_x$ would be an effective anti-canonical $\mathbb{Q}$-divisor not containing $C_x$ in its support and for which, by \cite[Lemma A.3]{CPW0}, the log pair $(S,D')$ is not log-canonical at $\msp$, which is impossible.   
\end{proof}

\medskip

Theorem \ref{thm:KKW24} together with several other existing results concerning sporadic cases in the classification of quasi-smooth well-formed del Pezzo hypersurfaces in weighted projective spaces provide quite strong evidence towards the following conjecture complementing Theorem \ref{thm:daiaj-cyl}:

\begin{conjecture} \label{conj:Cyl-criterion}
    Let $S\subset\mathbb{P}(a_0,a_1,a_2,a_3)$ be a quasi-smooth, well-formed del Pezzo hypersurface of degree $d$ defined over an algebraically closed field $\Bbbk$ of characteristic zero. Then $S$ admits an anti-canonical polar cylinder if and only if 
     $d = a_i + a_j$ for some $i\neq j$.
\end{conjecture}

\begin{remark}[{\cite[Example 6.4]{Saw}}] The conjecture is easily seen to be wrong if one relaxes the quasi-smoothness assumption. For instance the surface: 
$$S=\{t^2-x^2y^2+xz^3=0\}\subset \mathbb{P}(1,1,1,2)_{[x:y:z:t]}$$ 
is a normal del Pezzo surface of degree $2$ and Picard rank one, with two canonical singularities - $[1:0:0:0]$ of type $A_2$ and $[0:1:0:0]$ of type $A_5$ - 
which is not quasi-smooth and contains the anti-canonical polar cylinder $S\setminus \{x(t-xy)=0\}\cong (\mathbb{A}^1\setminus\{0\})\times \mathbb{A}^1$.    
\end{remark}

In view of Theorem  \ref{thm:KKW24} which covers all but finitely classes of surfaces belonging to the infinite series, the complete verification of Conjecture \ref{conj:Cyl-criterion} is reduced to studying the existence of anti-canonical polar cylinder for a finite number of classes of surfaces. Let us mention the following complementary result towards the above conjecture. 

\begin{theorem}[{\cite[Section 3]{CD}}]  \label{rem:CD} Conjecture \ref{conj:Cyl-criterion} holds for  quasi-smooth well-formed del Pezzo hypersurfaces $S\subset \mathbb{P}(a_0,a_1,a_2,a_3)$ of so-called  Brieskorn-Pham type, that is, defined by quasi-homogeneous equations of the form $x_0^{n_0}+x_1^{n_1}+x_2^{n_2}+x_3^{n_3}=0$. 
\end{theorem}


\subsubsection{Towards higher dimension}\label{4-2-4}
In this subsection, we review known results concerning the cylindricity of quasi-smooth Fano hypersurface $X$ in a four-dimensional weighted projective space with terminal (automatically $\mathbb{Q}$-factorial) singularities. In what follows $\Bbbk$ denotes  an algebraically closed field of characteristic zero. Let us first recall the  following useful lemma:

\begin{lemma}[{\cite[Proposition 1.5]{CPPZ}}]\label{prop:cppz}
Let $X$ be a normal  rationally chain connected projective variety of dimension $n \leqq 4$. Assume that $X$ contains cylinder $U= Z\times_{\Bbbk} \mathbb{A}^1_{\Bbbk}$ such that for every nonempty Zariski open subset $Z'$ of $Z$ the natural birational map $Z'\times_{\Bbbk}\mathbb{P}^1_{\Bbbk}\dashrightarrow X$ 
is not an open immersion.   Then $X$ is rational.  
\end{lemma}
\begin{proof}
Up to shrinking it if necessary, we can assume without loss of generality that $Z$ is smooth. Let $\bar{Z}$ be a smooth projective completion of $Z$ and  $\infty=\mathbb{P}^1_{\Bbbk}\setminus \mathbb{A}^1_{\Bbbk}$. The assumption implies that birational map $\varphi:\bar{Z}\times_{\Bbbk}\times \mathbb{P}^1_{\Bbbk} \dashrightarrow X$ induced by the inclusion $Z\times_{\Bbbk} \mathbb{A}^1_{\Bbbk}\subset X$ contracts the divisor $D_\infty:=\bar{Z}\times \{\infty\}$. It follows that $D_\infty$ is birationally $\mathbb{P}^1$-ruled, whence that $X$ is birational to $S\times_{\Bbbk}\mathbb{P}^1_{\Bbbk}\times_{\Bbbk}\mathbb{P}^1_{\Bbbk}$ for some smooth projective variety $S$ of dimension $n-2$. But since $X$ is rationally chain connected, $S$ is rationally chain connected as well, whence rational, which implies the rationality of $X$.   
\end{proof}


\begin{corollary}\label{cor:cppz}
Let $X$ be a ${\mathbb Q}$-factorial Fano variety with log-canonical singularities of Picard rank $\varrho (X)=1$ and of dimension $\dim X \leqq 4$. If $X$ is cylindrical, then $X$ is rational.
\end{corollary}
\begin{proof}
The variety $X$ is rationally chain connected by \cite{HM} and since it has Picard rank $1$, no cylinder in $X$ can be the restriction of a $\mathbb{P}^1$-cylinder, i.e., an open subset of the form $Z \times {\mathbb P}_{\Bbbk}^1$, contained in $X$. So, the rationality of $X$ follows from Lemma \ref{prop:cppz}.   
\end{proof}

\begin{remark} Note that even for smooth Fano threefolds of Picard rank $1$, rationality does not necessarily imply the existence of a cylinder. For instance though we know that {\it all} smooth prime Fano threefolds of genus $12$ are cylindrical (cf. \cite{KPZ1}), only smooth prime Fano threefolds $X_{2g-2}$ of genus $g=9, 10$ belonging to certain codimension one subset in the corresponding moduli are so far known to be cylindrical (cf. \cite{KPZ2}). 
\end{remark}


Quasi-smooth well-formed weighted hypersurface Fano threefolds are automatically $\mathbb{Q}$-factorial and have Picard rank $1$. Those with terminal singularities are completely classified and fall into 130 distinct families, see \cite{ABR} and \cite{BK}. 
 This classification provides a comprehensive framework for understanding their geometry, singularities, and birational properties. In particular, these families serve as natural testing grounds for problems concerning $K$-stability and the existence of cylinders. 
 
 \medskip 
 
 Among these Fano threefolds, for those of Fano index one which belong to 95 classified families, we have the following result: 

\begin{theorem}
 A quasi-smooth well-formed weighted hypersurface Fano threefold with terminal singularities and of Fano index one does not contain any cylinder. 
\end{theorem}
 \begin{proof} This follows from Corollary \ref{cor:cppz}
and the deeper fact that every every such Fano hypersurface is birationaly rigid \cite{CP16}, whence in particular not rational. 
 \end{proof}
 

Nevertheless, there exist rational quasi-smooth well-formed weighted Fano threefold hypersurfaces of Fano index one with non-terminal singularities. 
\begin{question}\label{prob:klt3fold} Does there exist cylindrical quasi-smooth well-formed weighted Fano threefold hypersurfaces of Fano index one ? 
\end{question}

\medskip

At the other side of the spectrum, among the remaining $130-95=35$ deformation families, there  exists $20$ distinguished families whose members are known to be all rational and have been verified to be also all cylindrical, namely: 

\begin{theorem}[{\cite[Theorem 1]{KKW252}}]\label{thm:kkw} With the notation of \cite{ACP}, the following hold:
\begin{enumerate}
    \item Every member belonging to the $8$ families
\begin{equation*}\label{subfamily}
\text{\textnumero~104, 105, 111, 113, 118, 119, 123, 126}
\end{equation*} 
contains the affine space $\mathbb{A}^3_{\Bbbk}$.
  \item Every member belonging to the $12$ families 
\begin{equation*}\label{list of Fano 3-folds}
    \text{\textnumero~ 106, 112, 114, 115, 120-122, 121, 124, 125, and 127–130}
\end{equation*}
\end{enumerate}
contains an open subset isomorphic to $({\mathbb A}_{\Bbbk}^1 \setminus \{ 0 \}) \times {\mathbb A}_{\Bbbk}^2$ but no open subset isomorphic to $\mathbb{A}^3_{\Bbbk}$.
\end{theorem}

\begin{example}\label{lem:106}
Assume for simplicity that $\Bbbk=\mathbb{C}$. A member $X$ of the family {\textnumero~106} 
equals:
\[
\{tw + f_4(x,y,z) =0\} \subseteq {\mathbb P}\big{(} 1, 1, 1, 2, 2 \big{)}_{[x:y:z:t:w]},
\]
for some homogeneous polynomial $f_4(x,y,z)$ of degree $4$ defining a smooth quartic curve $C$ in ${{\mathbb P}_{{\mathbb C}}^2}_{[x:y:z]}$.  The affine open subset $X \cap {\mathbb D}_+(w)$ is isomorphic to $\mathbb{A}_{{\mathbb C}}^3\big{/} {\mathbb Z}_2(1,1,1)$ which, by Proposition \ref{prop:adrien}, contains an open subset isomorphic to $({\mathbb A}_{{\mathbb C}}^1 \setminus \{ 0 \}) \times {\mathbb A}_{{\mathbb C}}^2$. On the other hand, $X$ does not contain any open subset isomorphic to $\mathbb{A}_{{\mathbb C}}^3$. Indeed, if such an open subset $U$ existed, then since $\mathrm{Pic}(U)$ is trivial, $X\setminus U$ would be the support of hyperplane section of $X$, say $H_x=\{x=0\}\cap X$ up to a linear coordinate change, so that $U$ is then isomorphic to the hypersurface in ${\mathbb A}_{{\mathbb C}}^4={\rm Spec}(\mathbb{C}[y,z,t,w])$ defined by the equation $ tw + f(1,y,z) =0$.  Noting that $U\setminus \{t=0\}\cong (\mathbb{A}_{{\mathbb C}}^1\setminus \{0\})\times \mathbb{A}_{{\mathbb C}}^1$ whereas $U\cap \{t=0\}\cong C_0\times \mathbb{A}_{\mathbb C}^1$ for some affine open subset $C_0$ of $C$, the topological Euler characteristic $\chi(U^{\mathrm{an}})$ of the underlying analytic space of $U$ then equals: $$\chi(U^{\mathrm{an}})=\chi(\mathbb{C}^*)+\chi(C_0^{\mathrm{an}})=\chi(C_0^{\mathrm{an}})< 1-2g(C) <-5,$$ where $g(C)=3$ is the genus of $C$, which is impossible since $(\mathbb{A}_{\mathbb C}^3)^{\mathrm{an}}=\mathbb{C}^3$ is contractible.   
\end{example}

The current state of knowledge for the remaining 15 families:  
\begin{equation*}\label{list of Fano 3-folds-15}
  \text{\textnumero~ 96–103, 107–110, 116, 117  and 122},
\end{equation*}
is little more contrasted. By  \cite{ok19}, {\it very general} members belonging to each of the $15$ families above  are not stably rational, hence are not cylindrical by Corollary \ref{cor:cppz}. In more detail, the following is known: {\it all} members in the families {\textnumero~96, 97, 98}, which are respectively cubic threefolds, quartic double solids and double cones of the Veronese surfaces, are classically known to be non-rational, hence non-cylindrical. On the other hand, {\it all } members of the families  {\textnumero~100, 101, 102, 103, 110} are known to be birationally solid by \cite{ok}, and hence neither rational nor cylindrical. In addition, {\it all} members belonging to the families  {\textnumero~107, 116} are known to be non-rational by \cite{pr1}, \cite{pr2}, whence non-cylindrical.  

The non-rationality of {\it all} members in the families {\textnumero~99, 108, 109, 117, 122} is still an open problem. Since non-cylindricity is less restrictive than non-rationality, the  following
question remains of independant interest: 
\begin{question}\label{question:non-cylindricity}
Is every member in the families in {\textnumero~99, 108, 109, 117, 122}  non-cylindrical ?
\end{question}
\section{Weighted Fano complete intersections}\label{section5}
In this section, all the varieties considered are assumed to be defined over an algebraically closed field $\Bbbk$ of characteristic zero. 

\subsection{Surfaces}\label{5-1}
In contrast to the weighted hypersurface case, the existence of weighted complete intersections admitting an anti-canonical polar cylinder is much more exceptional. In the setting of quasi-smooth, well-formed weighted complete intersection log del Pezzo surfaces of codimension at least two and index one, a complete classification is nowadays available (see \cite{KP15}). These surfaces are listed explicitly in Tables \ref{table:sporadic} and \ref{table:infinite}.

\begin{table}[htbp]
    \centering
    \caption{Sporadic cases}
    \label{table:sporadic}
    \rowcolors{2}{gray!15}{white}
    \begin{tabular}{@{}c|l|c || c|l|c@{} }
        \toprule
        \textnumero & $(a_0,a_1,a_2,a_3, a_4)$ & multidegree & \textnumero & $(a_0,a_1,a_2,a_3,a_4)$ & multidegree \\
        \midrule
        1  & $(1,2,2,3,3)$   & $(4,6)$   & 2  & $(1,2,3,4,5)$   & $(6,8)$ \\
        3  & $(1,3,3,5,5)$   & $(6,10)$  & 4  & $(1,4,5,7,11)$  & $(12,15)$ \\
        5  & $(1,4,7,10,13)$ & $(14,20)$ & 6  & $(1,5,8,12,19)$ & $(20,24)$ \\
        7  & $(1,5,9,13,17)$ & $(18,26)$ & 8  & $(1,7,11,17,27)$ & $(28,34)$ \\
        9  & $(1,7,12,17,23)$& $(24,35)$ & 10 & $(1,8,13,19,31)$ & $(32,39)$ \\
        11 & $(1,9,15,23,23)$& $(24,46)$ & 12 & $(2,2,3,3,3)$   & $(6,6)$ \\
        13 & $(2,3,4,5,5)$   & $(8,10)$  & 14 & $(2,3,5,6,7)$   & $(10,12)$ \\
        15 & $(3,3,5,5,7)$   & $(10,12)$ & 16 & $(3,5,6,8,13)$  & $(16,18)$ \\
        17 & $(3,5,7,9,11)$  & $(16,18)$ & 18 & $(4,5,7,10,13)$ & $(18,20)$ \\
        19 & $(5,7,10,14,23)$& $(28,30)$ & 20 & $(5,9,12,20,31)$& $(36,40)$ \\
        21 & $(5,14,17,21,37)$ & $(42,51)$ & 22 & $(6,7,9,11,14)$ & $(18,28)$ \\
        23 & $(6,8,9,11,13)$ & $(22,24)$ & 24 & $(9,15,23,23,31)$ & $(46,54)$ \\
        25 & $(9,15,23,23,37)$ & $(46,60)$ & 26 & $(9,23,30,38,67)$ & $(76,90)$ \\
        27 & $(10,17,25,34,43)$ & $(60,68)$ & 28 & $(11,18,27,44,61)$ & $(72,88)$ \\
        29 & $(11,27,36,62,97)$ & $(108,124)$ & 30 & $(11,29,38,48,85)$ & $(96,114)$ \\
        31 & $(11,29,39,49,59)$ & $(88,98)$ & 32 & $(11,29,39,49,67)$ & $(78,116)$ \\
        33 & $(13,22,55,76,97)$ & $(110,152)$ & 34 & $(13,23,34,56,89)$ & $(102,112)$ \\
        35 & $(13,23,35,47,57)$ & $(70,104)$ & 36 & $(13,23,35,57,79)$ & $(92,114)$ \\
        37 & $(14,19,25,32,45)$ & $(64,70)$ &  &  & \\
        \bottomrule
    \end{tabular}
\end{table}

\begin{table}[htbp]
    \centering
    \caption{Infinite series}
    \label{table:infinite}
    \rowcolors{2}{gray!15}{white}
    \begin{tabular}{@{}c|l|c@{} }
        \toprule
        \textnumero & $(a_0,a_1,a_2,a_3,a_4)$ & multidegree\\
        \midrule
        1 & $(1,1,n,n,2n-1)$ & $(2n,2n)$\\
        2 & $(1,2,2n+1,2n+1,4n+1)$ & $(4n+2,4n+3)$\\
        3 & $(2,2n+1,2n+1,4n+1,6n+1)$ & $(6n+3,8n+2)$\\
        \bottomrule
        \rowcolor{white}
        \multicolumn{3}{l}{\footnotesize $n$ denotes a positive integer.}
        \end{tabular}
    \end{table}
\begin{theorem}[{\cite{KP15}, \cite{KW19}}]\label{thm:ldp}
    For every quasi-smooth, well-formed weighted complete intersection log del Pezzo surface of codimension at least two and index one, the $\alpha$-invariant is at least $1$, except for the following two special cases:
    \begin{itemize}
        \item weighted complete intersections of multidegree $(2n,2n)$ in $\mathbb{P}(1,1,n,n,2n-1)$ for positive integers $n$, and
        \item certain weighted complete intersections of multidegree $(6,8)$ in $\mathbb{P}(1,2,3,4,5)$ where the defining equation of degree $6$ does not contain the monomial $yt$ (with $y$ of weight $2$ and $t$ of weight $4$).
    \end{itemize}
\end{theorem}
By Theorem \ref{theorem:alpha cylinder}, it follows that any quasi-smooth, well-formed weighted complete intersection log del Pezzo surface of codimension at least two and index one, with the exception of the two special cases mentioned above, does not contain an anti-canonical polar cylinder.
\begin{remark}\label{rem:ldp}
The quasi-smoothness condition in the assertion of Theorem \ref{thm:ldp} can not be relaxed. For instance, we will consider the following weighted complete intersection log del Pezzo surface:
\[
S=S_{4,6}= \big{\{} \, xv-yz= tv+z^3+y^3+x^6=0 \, \big{\}} \subseteq {\mathbb P} \big{(} 1, 2, 2, 3, 3 \big{)}_{[x:y:z:t:v]}, 
\]
which is well-formed but not quasi-smooth and $-K_S \sim_{\mathbb Q} {\mathscr O}_S(1)$. Then the complement of the divisor:
\[
D\coloneqq {\big{\{} zt =0 \big{\}}}|_S \sim_{\mathbb Q} 5(-K_S)
\]
is isomorphic to $S \setminus D \cong \big{(} {\mathbb A}_{\Bbbk}^1 \setminus \{ 0 \} \big{)} \times {\mathbb A}_{\Bbbk}^1$. In particular, $S$ contains an anti-canonical polar cylinder, which implies in turn that $\alpha (S) <1$. 
\end{remark}
\subsection{Cylindricity in Higher Dimensions}\label{5-2}
In low dimensions, cylindrical quasi-smooth, well-formed weighted complete intersections of higher codimension are extremely rare even if we drop the Fano index one condition. However, in higher dimensions, more precisely, for dimension greater than or equal to four, such examples do exist and can be constructed explicitly. 
\begin{theorem}\label{thm:projection of complete intersection}
    Let $X\subset \mathbb{P}(a_0, \ldots, a_n)_{[x_0:\cdots :x_n]}$ be a quasi-smooth, well-formed weighted complete intersection of multidegree $(d_1, d_2)$ for $n \geqq 6$. If there exist weights $a_i$ and $a_j$ such that: 
    \begin{equation*}
         d_1 = a_i + a_{i_1} = a_j + a_{j_1}\qquad d_2 = a_i + a_{i_2} = a_j + a_{j_2}
    \end{equation*}
     for some indices $i_1,i_2,j_1,j_2$ with the set $\{i, j, i_1, i_2, j_1, j_2\}$ consisting of six distinct elements, then $X$ contains an anti-canonical polar cylinder. More precisely, $X$ contains an ${\mathbb A}^{n-5}$-cylinder.
\end{theorem}
\begin{proof}
    For convenience, we may assume that $(i,j) = (0,1)$. By a reasoning similar to that of the proof of  Theorem \ref{thm:daiaj-cyl} (recall that the ground field $\Bbbk$ in this section is assumed to be algebraically closed), we can assume up to a suitable coordinate change that $X$ is defined by two quasi-homogeneous polynomials: 
    \begin{equation*}
        x_0x_{i_1} + x_1x_{j_1} + f(x_2,\ldots, x_n) = x_0x_{i_2} + x_1x_{j_2} + g(x_2,\ldots, x_n) = 0,
    \end{equation*}
    where $f(x_2,\ldots, x_n)$ and $g(x_2,\ldots, x_n)$ are quasi-homogeneous polynomials in the variables $x_2,\ldots, x_n$ of degrees $d_1$ and $d_2$, respectively. Consider the projection:
    \begin{equation*}
        \pi\colon X\dashrightarrow  \PP(a_{2},\ldots,a_{n})_{[x_2:\cdots :x_n]},\qquad [x_0:\cdots:x_n]\longmapsto [x_2:\cdots:x_n].
    \end{equation*}
    Let $C\subset \PP(a_2,\ldots,a_n)$ be the subvariety defined by $x_{i_1}x_{j_2} - x_{i_2}x_{j_1} = 0$. On the complement $\PP(a_2,\ldots,a_n)\setminus C$, the projection $\pi$ admits an inverse morphism. Notice that:
    \[
    \PP (a_2, \ldots, a_n) \setminus C \cong {\rm Spec} \Big{(} 
    \Bbbk {\Big{[} x_2, \ldots , x_n, \frac{1}{x_{i_1}x_{j_2}-x_{i_2}x_{j_1}} \Big{]}}_0 \Big{)}
    \]
    \[
    \cong 
{\rm Spec} \Big{(} 
    \Bbbk {\Big{[} x_{i_1},x_{i_2},x_{j_1},x_{j_2}, \frac{1}{x_{i_1}x_{j_2}-x_{i_2}x_{j_1}} \Big{]}}_0 \Big{)} \times {\mathbb A}_{\Bbbk}^{n-5}. 
    \]
    Consequently, this construction shows that $X$ contains an ${\mathbb A}^{n-5}$-cylinder. Note that $\pi$ gives rise to an isomorphism between $X \setminus \Delta$ and ${\mathbb P}(a_2, \cdots , a_n) \setminus C$, where:
    \[
    \Delta \coloneqq \big{\{} \, x_{i_1} x_{j_2} - x_{i_2}x_{j_1}=0 \, \big{\}} \cap  X \sim_{{\mathbb Q}} \frac{d_1+d_2-(a_0+a_1)}{ \big{(} \sum_{l \neq 0,1,i_1, j_2} a_i \big{)}} \Big{(} -K_X \Big{)} 
    \]
Thus it follows that this cylinder is actually anti-canonically polar.    
\end{proof}
\begin{remark}
    We can construct an explicit example that satisfies the condition stated in Theorem \ref{thm:projection of complete intersection}. Consider the weighted projective space: 
    \begin{equation*}
        \mathbb{P}\coloneqq\mathbb{P}(\underbrace{1,\ldots,1}_{m\ \text{times}}, 3, 3, 9, 11),
    \end{equation*}
    where the weight $1$ occurs $m\geq 2$ times. Then there exists a quasi-smooth, well-formed weighted complete intersection $X\subset\mathbb{P}$ of multidegree $(12,12)$. By Theorem \ref{thm:projection of complete intersection}, it follows that $X$ contains a cylinder.
\end{remark}
The above argument allows us to construct infinitely many families containing quasi-smooth, weighted complete intersections of dimension greater than or equal to $4$, which are not intersections with linear cones. This naturally leads to the following question:
\begin{question}
    Does there exist quasi-smooth, well-formed weighted complete intersection threefolds that contain a cylinder other than complete intersections of two quadric hypersurfaces in ${\mathbb P}_{\Bbbk}^5$? \footnote{Note that any smooth complete intersection of two quadrics in ${\mathbb P}_{\Bbbk}^5$ is cylindrical, see \cite[Proposition 5.0.1.]{KPZ1}.}
\end{question}
The similar proof as for Theorem \ref{thm:projection of complete intersection} allows to deduce the following generalization. Since the proof is a straightforward generalization of that for Theorem \ref{thm:projection of complete intersection}, we will leave the detailed proof to the readers:  
\begin{theorem}\label{thm:wci}
Let
\[
X=X_{d_1, \ldots , d_c} \subseteq {\mathbb P} \big{(}a_0, \ldots, a_n \big{)}_{[x_0:\cdots :x_n]}
\]
be a quasi-smooth, well-formed complete intersection of codimension $c$ of multi-degree $(d_1, \ldots , d_c)$ with $n\geqq c(c+1)$. If there  exist mutually distinct indices:
\[
i^{(1)}, \ldots, i^{(c)}, i_1^{(1)}, \ldots , i_1^{(c)}, \ldots  , 
i_c^{(1)}, \ldots , i_c^{(c)}
\]
such that:
\[
d_j=a_{i^{(1)}}+a_{i_j^{(1)}}= \cdots = a_{i^{(c)}}+a_{i_j^{(c)}} \quad (j=1, \ldots , c), 
\]
then $X$ is anti-canonically polar cylindrical, more precisely, $X$ contains an ${\mathbb A}^{[n+1-c(c+1)]}$-cylinder. 
\end{theorem}

\end{document}